\documentclass[11pt,oneside,english]{amsart}
\usepackage[T1]{fontenc}
\usepackage[latin9]{inputenc}
\usepackage{geometry}
\usepackage{textcomp}
\usepackage{amsbsy}
\usepackage{amstext}
\usepackage{amsthm}
\usepackage{amssymb}
\usepackage{wasysym}

\makeatletter
\numberwithin{equation}{section}
\numberwithin{figure}{section}
\theoremstyle{plain}
\newtheorem{thm}{\protect\theoremname}[section]
\theoremstyle{plain}
\newtheorem{lem}[thm]{\protect\lemmaname}
\theoremstyle{plain}
\newtheorem{prop}[thm]{\protect\propositionname}
\theoremstyle{definition}
\newtheorem{example}[thm]{\protect\examplename}
\theoremstyle{remark}
\newtheorem{rem}[thm]{\protect\remarkname}

\def\makebbb#1{
    \expandafter\gdef\csname#1\endcsname{
        \ensuremath{\Bbb{#1}}}
}
\makebbb{R}
\makebbb{N}
\makebbb{Z}
\makebbb{C}
\makebbb{G}
\makebbb{Q}
\makebbb{H}
\makebbb{P}
\makebbb{B}
\makebbb{K}
\makebbb{E}

\makeatother

\usepackage{babel}
\providecommand{\examplename}{Example}
\providecommand{\lemmaname}{Lemma}
\providecommand{\propositionname}{Proposition}
\providecommand{\remarkname}{Remark}
\providecommand{\theoremname}{Theorem}

\begin{document}
\title{Convergence rates for discretized Monge-Ampère equations and quantitative
stability of optimal transport }
\author{Robert J. Berman}
\begin{abstract}
In recent works - both experimental and theoretical - it has been
shown how to use computational geometry to efficently construct approximations
to the optimal transport map between two given probability measures
on Euclidean space, by discretizing one of the measures. Here we provide
a quantitative convergence analysis for the solutions of the corresponding
discretized Monge-Ampère equations. This yields $H^{1}-$converge
rates, in terms of the corresponding spatial resolution $h,$ of the
discrete approximations of the optimal transport map, when the source
measure is discretized and the target measure has bounded convex support.
Periodic variants of the results are also established. The proofs
are based on new quantitative stability results for optimal transport
maps, shown using complex geometry.
\end{abstract}

\address{Robert J. Berman, Mathematical Sciences, Chalmers University of Technology
and the University of Gothenburg, SE-412 96 Göteborg, Sweden, tel:
+46(0)735629321}
\email{robertb@chalmers.se}
\maketitle

\section{Introduction}

The theory of \emph{optimal transport} \cite{v1}, which was originally
motivated by applications to logistics and economics, has generated
a multitude of applications ranging from meteorology and cosmology
to image processing and computer graphics in more recent years \cite{sa,so}.
This has led to a rapidly expanding literature on numerical methods
to construct optimal transport maps, using an appropriate discretization
scheme. From the PDE point of view this amounts to studying discretizations
of the second boundary value problem for the Monge-Ampère operator.
The present paper is concerned with a particular discretization scheme,
known as \emph{semi-discrete optimal transport} in the optimal transport
literature (see \cite{b-d} and references therein for other discretization
schemes, based on finite differences). This approach uses computational
geometry to compute a solution to the corresponding discretized Monge-Ampère
equation and exhibits remarkable numerical performance, using a damped
Newton iteration \cite{me,l-s}. The convergence of the iteration
towards the discrete solution was recently settled in \cite{k-m-t}
and one of the main aims of the present paper is to establish quantitative
convergence rates of the discrete solutions, as the spatial resolution
$h$ tends to zero. 

\subsection{\label{subsec:Background intro}Background }

Through out the paper we fix open bounded domains $X$ and $Y$ in
$\R^{n}$ with $Y$ assumed convex and a probability measure $\nu$
on $Y$ with a density which is uniformly bounded from below:
\begin{equation}
\nu=1_{Y}g(y)dy,\,\,\,\,g\in L^{1}(\R^{n}),\,\,\delta:=\inf_{Y}g>0\label{eq:def of nu}
\end{equation}
We recall that in the case when $\mu$ is a probability measure on
$X$ which is also absolutely continuous with respect to $dx,$ i.e.
$\mu\in\mathcal{P}_{ac}(X),$ then a map $T$ in $L^{\infty}(X,Y)$
is said to be a\emph{ transport map} (with \emph{source} $\mu$ and
\emph{target} $\nu)$ if 
\[
T_{*}\mu=\nu
\]
 and $T$ is said to be an\emph{ optimal transport map }(with respect
to the Euclidean cost function $|x-y|^{2}),$\emph{ }denoted by $T_{\mu}^{\nu},$
if it realizes the infimum defining the Wasserstein $L^{2}-$distance
$W_{2}(\mu,\nu)$ \cite{v1}: 
\[
W_{2}(\mu,\nu):=\inf_{T}\int_{X}|x-Tx|^{2}\mu,
\]
 where the infimum ranges over all transport maps. By Brenier's theorem
\cite{br} there exists a unique optimal transport map $T_{\mu}^{\nu}$
and it has the characteristic property of being a gradient map:
\[
T_{\mu}^{\nu}=\nabla\phi
\]
(in the almost everywhere sense) for a convex function $\phi$ on
$X,$ called the\emph{ potential }of $T_{\mu}^{\nu}$ (see \cite{v1}
for further background on optimal transport theory). The potential
$\phi$ is the unique (modulo an additive constant) convex solution
to the corresponding \emph{second boundary value problem} for the
\emph{Monge-Ampère operator}: the sub-gradient image of $\phi$ is
contained in the closure of $Y,$ 
\begin{equation}
(\partial\phi)(X)\subseteq\bar{Y}\label{eq:sec bd cond intro}
\end{equation}
and $\phi$ solves the equation 
\begin{equation}
MA_{g}(\phi)=\mu,\label{eq:ma eq with mu intro}
\end{equation}
where the Monge-Ampère measure $MA_{g}$ is the probability measure
on $X$ defined by
\begin{equation}
MA_{g}(\phi):=g(\nabla\phi)\det(\nabla^{2}\phi)dx\label{eq:formul for MA g intro}
\end{equation}
when $\phi$ is $C^{2}-$smooth and the general definition, due to
Alexandrov, is recalled in Section \ref{subsec:Recap-of-the}. We
will say that $(X,Y,\mu,\nu)$ is\emph{ regular} if the corresponding
solution $\phi$ is in $C^{2}(\bar{X)}.$ By Cafferelli's regularity
results \cite{c1,ca0,c2} this is the case if $X$ and $Y$ are assumed
strictly convex with $C^{2}-$boundary and the densities of $\mu$
and $\nu$ are Hölder continuous and strictly positive on $\bar{X}$
and $\bar{Y},$ respectively. Then $\phi$ defines a classical solution
of the corresponding PDE and the corresponding optimal transport map
$\nabla\phi$ yields a diffeomorphism between the closures of $X$
and $Y.$ In fact, as is well-known, for \emph{any} probability measure
$\mu$ there exists a solution (in the weak sense of Alexandrov) of
the corresponding second boundary value problem, which is uniquely
determined up to normalization (Lemma \ref{lem:uniqueness and extension}).
In the sequel it will be convenient to use the normalization condition
that the integral of a solution over $(X,dx)$ vanishes. 

\subsubsection{\label{subsec:Discretization-using-semi-discre}Discretization using
semi-discrete transport}

A time-honored approach for discretizing Monge-Ampère equations, which
goes back to the classical work of Alexandrov on Minkowski type problems
for convex bodies and polyhedra, amounts to replacing the given probability
measure $\mu$ with a sequence of discrete measures converging weakly
towards $\mu$ (see \cite[Thm 7.3.2 and Section 7.6.2]{al} and \cite[Section 17]{ba}).
A standard way to obtain such a sequence is to first discretize $X$
by fixing a sequence of ``point clouds'' $(x_{1},...,x_{N})\in X^{N}$
and a dual tessellation of $X$ with $N$ cells $(C_{i})_{i=1}^{N}.$
This means that the union of $C_{i}$ cover $X,$ $x_{i}\in C_{i}$
and the intersection of different cells have zero Lebesgue measure.
For example, given a point cloud the corresponding Voronoi tessellation
of $X$ provides a canonical dual tessellation of $X$. The ``spatial
resolution'' of the discretization is quantified by
\[
h:=\max_{i\leq N}\text{diam\ensuremath{(C_{i}),}}
\]
where $\text{diam}\ensuremath{(C_{i}),}$ denotes the diameter of
the cell $C_{i}.$ The corresponding discretization of the measure
$\mu$ is then defined by setting
\begin{equation}
\mu_{h}:=\sum_{i=1}^{N}f_{i}\delta_{x_{i}},\,\,\,\,f_{i}:=\mu(C_{i})\label{eq:def of mu h}
\end{equation}
 where we have used the subindex $h$ to emphasize that we are focusing
on the limit when $h\rightarrow0$ (see also Section \ref{subsec:Other-discretization-schemes}
for other discretizations). This discretization scheme corresponds,
from the point of view of optimal transport, to the notion of \emph{semi-discrete
}optimal transport (since it corresponds to optimal transport between
the ``continuous'' measure $\nu$ and the discrete measure $\mu_{h};$
see \cite{k-m-t,l-s} and Section \ref{subsec:Comparison-with-semi-discrete}).
From the point of view of numerics this kind of discretization scheme
was first introduced in the different setting of the Dirichlet problem
in \cite{o-p}.

\subsection{\label{subsec:Convergence-rates-for}Convergence rates for discretized
Monge-Ampère equations}

Given a point-cloud on $X$ with spatial resolution $h$ we denote
by $\phi_{h}$ the normalized convex solution to the corresponding
Monge-Ampère equation \ref{eq:ma eq with mu intro} with right-hand
side given by the discrete measure $\mu_{h}.$ It follows from a standard
general convexity argument that, in the limit when $h\rightarrow0$
(and hence $N\rightarrow\infty)$ the functions $\phi_{h}$ converge
uniformly towards the solution $\phi$ and, as a consequence, the
gradients of $\phi_{h}$ converge weakly towards the gradient of $\phi$
(in the sense of distributions; see Prop \ref{prop:qual stab}). But
the argument gives no control on the rate of convergence (in terms
of $h$ or $N^{-1}$$)$ and the main purpose of the present work
is to provide such a result: 
\begin{thm}
\label{thm:thm H1 intro}\emph{(Regular case)} Assume that $(X,Y,\mu,\nu)$
is regular and let $\mu_{h}$ be a discretization of $\mu.$ Denote
by $\phi$ and $\phi_{h}$ the solutions to the corresponding second
boundary value problems for the Monge-Ampère operator. There exists
a constant $C_{1}$ (depending on $(X,Y,\mu,\nu))$ such that
\begin{equation}
\left\Vert \phi_{h}-\phi\right\Vert _{H^{1}(X)}:=\left(\int_{X}|\nabla\phi_{h}-\nabla\phi|^{2}dx\right)^{1/2}\leq C_{1}h^{1/2}\label{eq:estimate on H one in thm intr}
\end{equation}
More precisely, if $\nabla^{2}\phi\geq C_{0}^{-1}I,$ then
\begin{equation}
C_{1}=\sqrt{n(n+1)C_{0}^{n-1}\delta^{-1}}\label{eq:formula for C one in first thm intro}
\end{equation}
where $\delta$ is the constant in formula \ref{eq:def of nu}). The
result thus holds more generally as long as $\phi$ is uniformly convex
and $\delta>0.$ 
\end{thm}

As a consequence, if $C_{P}$ denotes the constant in the $L^{2}-$Poincaré
inequality on $X$ (i.e. $\left\Vert u\right\Vert _{L^{2}(X)}\leq C_{P}\left\Vert \nabla u\right\Vert _{L^{2}(X)}$
for $u$ with zero average), then 
\begin{equation}
\left\Vert \phi_{h}-\phi\right\Vert _{L^{2}(X)}\leq C_{P}C_{1}h^{1/2}\label{eq:L two estimate thm H1 intro}
\end{equation}
From the point of view of optimal transport theory the previous theorem
says that the optimal transport map $\nabla\phi$ from $\mu$ to $\nu$
(defining a diffeomorphism between the closures of $X$ and $Y)$
may be quantitatively approximated by the $L^{\infty}-$maps $\nabla\phi_{h}.$
As explained in Section \ref{sec:Formulation-in-terms} the gradient
maps $\nabla\phi_{h}$ are piecewise constant on the convex hull of
the corresponding point cloud (more precisely, $\nabla\phi_{h}$ is
constant on the facets of the weighted Delaunay tesselation of $\R^{n}$
induced by $\phi_{h}).$ 

We will also establish the following universal bound which applies
in the general case. It yields, in particular, a quantitative approximation
of the corresponding optimal transport map if $\mu$ has a density.
\begin{thm}
\label{thm:(General-case) intro}\emph{(General case)} Let $X$ and
$Y$ be bounded domains in $\R^{n}$ with $Y$ assumed convex and
$\nu$ a probability measure of the form \ref{eq:def of nu}. Then,
for any given probability measure $\mu$ on $X,$ 
\[
\left\Vert \phi_{h}-\phi\right\Vert _{H^{1}(X)}\leq C_{2}h^{1/2^{n}},
\]
 where the constant $C_{2}$ only depends on upper bounds on the diameters
of $X$ and $Y$ and on positive lower bounds on the volume $V(Y)$
of $Y$ and $\delta(:=\sup_{Y}(\nu/dy)).$
\end{thm}

The previous theorems are obtained as straightforward consquences
of the following analytic inequalities of independent interest (see
Theorem \ref{thm:analytic ineqs text} for the precise statements).
Let $\phi_{0}$ and $\phi_{1}$ be two convex functions on $X$ whose
sub-gradient images are equal to $\bar{Y}$ (without of loss of generality
the functions may be assumed to be smooth). Then 

\begin{equation}
\int_{X}|\nabla\phi_{0}-\nabla\phi_{1}|^{2}dx\leq C_{0}\int_{X}(\phi_{1}-\phi_{0})\left(MA_{g}(\phi_{0})-MA_{g}(\phi_{1})\right),\label{eq:analytic ineq reg intro}
\end{equation}
for a constant $C_{0}$ depending on a positive lower bound on $\nabla^{2}\phi_{0}.$
Moreover, in general,
\begin{equation}
\int_{X}|\nabla\phi_{0}-\nabla\phi_{1}|^{2}dx\leq C\left(\int_{X}(\phi_{1}-\phi_{0})\left(MA_{g}(\phi_{0})-MA_{g}(\phi_{1})\right)\right)^{1/2^{n-1}},\label{eq:analytic ineq general intro}
\end{equation}
for a constant $C$ independent of $\phi_{0}$ and $\phi_{1}.$ Using
that 
\begin{equation}
\int_{X}(\phi_{1}-\phi_{0})\left(MA_{g}(\phi_{0})-MA_{g}(\phi_{1})\right)\leq d(Y)W_{1}\left(MA_{g}(\phi_{0}),MA_{g}(\phi_{1})\right),\label{eq:I smaller than Wone intro}
\end{equation}
 where $W_{1}$ denotes the Wasserstein $L^{1}-$distance on $\mathcal{P}(X)$
and $d(Y)$ the diameter of $Y,$ Theorems \ref{thm:thm H1 intro},
\ref{thm:(General-case) intro} are deduced by setting $\phi_{0}=\phi$
and $\phi_{1}=\phi_{h},$ and noting that the right hand side in formula
\ref{eq:I smaller than Wone intro} is bounded from above by $h.$ 

\subsection{Quantitative stability for optimal transport maps}

The combination of the inequalities \ref{eq:analytic ineq reg intro},
\ref{eq:analytic ineq general intro} with the inequality \ref{eq:I smaller than Wone intro}
may be succintly formulated as a quantitative stability result for
optimal transport maps: 
\begin{thm}
\label{thm:(Quantitative-stability--first intro}\emph{(Regular case)}.
Assume that $(X,Y,\mu_{0},\nu)$ is regular. Then there exists a constant
$c_{1},$ depending on $\mu_{0}$ and $\nu,$ such that
\[
\left\Vert T_{\mu_{0}}^{\nu}-T_{\mu_{1}}^{\nu}\right\Vert _{L^{2}(X,dx)}\leq c_{1}W_{1}(\mu_{0},\mu_{1})^{1/2}
\]
 for any $\mu_{1}\in\mathcal{P}_{ac}(X).$ More generally, if the
potential $\phi$ of $T_{\mu_{0}}^{\nu}$ is uniformly convex, then
the inequality above holds with $c_{1}=C_{1}\sqrt{d(Y)},$ where $d(Y)$
denotes the diameter of $Y$ and $C_{1}$ is as in formula \ref{eq:formula for C one in first thm intro}
. Moreover, if $\mu_{0}$ and $\mu_{1}$ have densities in $L^{2}(X),$
then
\[
\left\Vert T_{\mu_{0}}^{\nu}-T_{\mu_{1}}^{\nu}\right\Vert _{L^{2}(X,dx)}\leq c_{1}C_{P}\left\Vert \frac{\mu_{0}}{dx}-\frac{\mu_{1}}{dx}\right\Vert _{L^{2}(X,dx)},
\]
 where $C_{P}$ is the Poincaré constant on $X.$
\end{thm}

Recall that here $T_{\mu}^{\nu}$ denotes the $L^{2}-$optimal transport
map from $\mu$ to $\nu.$ The power $1/2$ in the first inequality
of the previous theorem is sharp, as explained in Section \ref{subsec:Sharpness-of-the}.
In the general case we have:
\begin{thm}
\label{thm:(Quantitative-stability--second intro}\emph{(General case)}.
There exists a constant $c_{2}$ such that 
\[
\left\Vert T_{\mu_{0}}^{\nu}-T_{\mu_{1}}^{\nu}\right\Vert _{L^{2}(X,dx)}\leq c_{2}W_{1}(\mu_{0},\mu_{1})^{1/2^{n}}
\]
 for any pair $\mu_{0},\mu_{1}\in\mathcal{P}_{ac}(X).$ More precisely,
$c_{2}$ only depends on upper bounds on the the diameters of $X$
and $Y$ and on positive lower bounds on the volume $V(Y)$ of $Y$
and $\delta(:=\sup_{Y}(\nu/dy))$ Moreover, if $\mu_{0}$ and $\mu_{1}$
have densities in $L^{2}(X),$ then 
\[
\left\Vert T_{\mu_{0}}^{\nu}-T_{\mu_{1}}^{\nu}\right\Vert _{L^{2}(X,dx)}\leq c_{3}\left\Vert \frac{\mu_{0}}{dx}-\frac{\mu_{1}}{dx}\right\Vert _{L^{2}(X,dx)}^{1/(2^{n}-1)}
\]
 for a constant $c_{3}$ dependening on the same quantitites as the
constant $c_{2}.$
\end{thm}

Since $W_{1}\leq W_{p},$ where $W_{p}$ denotes the $L^{p}$$-$Wasserstein
distance, the previous Theorems also hold with $W_{1}$ replaced by
$W_{p}.$ 

Note that Theorem \ref{thm:(Quantitative-stability--first intro}
and Theorem \ref{thm:(Quantitative-stability--second intro} imply
Theorem \ref{thm:thm H1 intro} and Theorem \ref{thm:(General-case) intro},
respectively. Indeed, even though the Monge-Ampère measure of $\phi_{h}$
is discrete and hence the push-forward $(\nabla\phi_{h})_{*}\mu_{h}$
is ill-defined, one can first apply Theorems \ref{thm:(Quantitative-stability--first intro},
\ref{thm:(Quantitative-stability--second intro} to a regularization
$\phi_{h}^{\epsilon}$ of $\phi_{h}$ (e.g. obtained by convolution)
and then let $\epsilon\rightarrow0.$ Anyhow, as explained above,
all theorems above will be deduced from the analytic inequalities
\ref{eq:analytic ineq reg intro}, \ref{eq:analytic ineq general intro}.

\subsection{Relations to computational geometry}

An important motivation for the present work comes from the recent
result in \cite{k-m-t}, showing that the vector $\boldsymbol{\phi}_{h}:=(\phi_{h}(x_{1}),...\phi_{h}(x_{N}))\in\R^{N}$
- which solves a discrete variant of the Monge-Ampère equation (with
target $Y)$ - may be effectively computed using a damped Newton iteration
on $\R^{N},$ which converges globally at a linear rate towards $\boldsymbol{\phi}_{h}$
(and locally at a quadratic rate if $g$ is Lipschitz continuous).
The iteration is defined in terms of computational geometry and the
restriction of the solution $\phi_{h}$ to the convex hull $X_{h}$
of the points $\{x_{1},...,x_{N}\}$ can then be recovered as the
piecewise affine function defined by convex hull of the discrete graph
of $\boldsymbol{\phi}_{h}.$ From this computational point of view
Theorems \ref{thm:thm H1 intro}, \ref{thm:(General-case) intro}
above yield quantitative convergence results for the corresponding
discrete objects defined on $X_{h}$ in the ``continuous limit''
when $h\rightarrow0.$ This is explained in Section \ref{sec:Formulation-in-terms}.

\subsection{The periodic setting }

Now assume that $\mu$ is a given $\Z^{n}-$periodic measure on $\R^{n}$
normalized so that its total mass on a (or equivalently any) fundamental
region $X$ is equal to one (e.g. $X=[0,1[^{n}).$ We then consider
the corresponding Monge-Ampère equation \ref{eq:ma eq with mu intro}
for a convex function $\phi$ on $\R^{n}$ subject to the condition
that $\partial\phi$ be periodic (which replaces the second boundary
condition \ref{eq:sec bd cond intro}). Such a convex function will
be called\emph{ quasi-periodic.}

In this periodic setting Theorem \ref{thm:thm H1 intro} still applies
(with $X=[0,1[^{n})$ (see Section \ref{sec:The-periodic-setting}).
In terms of optimal transport the induced diffeomorphism $T_{\phi}$
of the torus $(\R/\Z)^{n}$ transporting $\mu$ to the Riemannian
volume form on the flat torus is optimal with respect to the cost
function $d(x,y)^{2}$ where $d$ denotes the Riemannian distance
function on the flat torus \cite{c-e,hu}. 

\subsection{Proofs by complexification}

The proofs of the key analytic inequalities \ref{eq:analytic ineq reg intro}
and \ref{eq:analytic ineq general intro} use a complexification argument
to first deduce the special case when $g$ is constant from well-known
inequalities in Kähler geometry and pluripotential theory (due to
Aubin \cite{aub} and Blocki \cite{bl}, respectively). Then a separate
variational argument is used to reduce the case of a general $g$
to the special case of a constant $g.$ The universal dependence in
Theorem \ref{thm:(General-case) intro} is obtained by exploiting
that a convex body $Y$ induces a canonical toric Kähler-Einstein
metric, whose analytical properties are controlled by upper bounds
on the diameter of $Y$ and the inverse of the volume of $Y$ (thanks
to the estimates in \cite{do,kl}). 

One may wonder whether the use of complex geometry is really necessary?
In a nutshell, the complexification method has two advantages:
\begin{itemize}
\item The exterior algebra of positive forms can be can be leveraged
\item By a compactification argument (involving toric varieites) one may
perform integration by parts without boundary terms. 
\end{itemize}
(the second point is not needed in the periodic setting). The first
point can without doubt be replaced by a a suitable linear algebra
of mixed real Monge-Ampère operators and mixed discriminents, etc.
However, at least to the author, this appears to make the calculations
rather unwieldy (but see \cite{la} for a real formalism mimicking
the complex formalism). Moreover, it seems likely that the second
point can be circumvented by an appropriate choice of cut-off functions
in $\R^{n}.$ Anyhow, an important merit of the general complexifixation
method employed in this paper is that it also opens the door for direct
applications of other results in complex geometry to the second boundary
value problem for the real Monge-Ampère equation and optimal transport.
This method is, in spirit, similar to Gromov's approach \cite{gro}
to the Brunn-Minkowski and Alexander-Fenchel inequalities for convex
bodies, which also exploits the complex geometry of toric varieties.

\subsection{\label{subsec:Comparison-with-previous}Comparison with previous
results}

There seem to be no prior results giving convergence rates (in any
norms) for the functions $\phi_{h}$ or the vectors $\boldsymbol{\phi}_{h}$
in the limit when $h\rightarrow0$ (even in the model case of a uniform
grid on the torus). Another quantitative stability result for optimal
transport maps, in the regular setting, has previously been obtained
by Ambrosio (reported in \cite{gi} where it was shown to be sharp).
Translated into the present notation the result in \cite{gi} shows
that the inequality in Theorem \ref{thm:(Quantitative-stability--first intro}
holds if the $L^{\infty}-$mappings $T_{\mu_{i}}^{\nu}$ are replaced
by their inverse, i.e. by the optimal transport maps $T_{\nu}^{\mu_{i}}$
and the $L^{1}-$Wasserstein distance $W_{1}$ is replaced with the
(weaker) $L^{2}-$Wasserstein distance $W_{2}$ (see \cite[Cor 3.4]{gi}
and its proof). In other words, while Theorems \ref{thm:(Quantitative-stability--first intro},
\ref{thm:(Quantitative-stability--second intro} above concern stability
of optimal transport maps with respect to variations of the \emph{source
measure,} the result \cite[Cor 3.4]{gi} concerns variations with
respect to the \emph{target measure.} Since $T_{\nu}^{\mu}=\nabla\phi^{*},$
where $\phi^{*}$ denotes the Legendre transform of the solution $\phi$
to the second boundary value problem discussed above, the stability
result in \cite[Cor 3.4]{gi} can be used to obtain the same rates
$O(h^{1/2})$ for the $H^{1}-$norm of the difference of Legendre
transforms $\phi_{h}^{*}-\phi^{*}.$ 

In the different setting of the Oliker- Prussner discretization of
the \emph{Dirichlet problem} for the Monge-Ampère equation \cite{o-p}
with $g=1$ (where the second boundary condition \ref{eq:sec bd cond intro}
is replaced by the vanishing of the solution $\phi$ at the boundary
of $X)$ convergence rates for $L^{\infty}-$norms were established
in \cite{n-z1} at a rate $O(h^{\alpha})$ under the assumption that
the solution $\phi\in C^{2,\alpha}(X).$ The proofs are based on a
combination of discrete Alexandroff $L^{\infty}-$estimates with Brunn-Minkowski
inequalites (see also the recent work \cite{c-h-w} where an optimal
rate $O(h^{2})$ is shown, assuming bounds on the 4th order derivatives
of $\phi).$ Moreover, in \cite{n-z2} the $L^{\infty}-$rates in
\cite{n-z1} and \cite{c-h-w} were used to obtain rates for $W_{p}^{2}-$norms. 

\subsection{\label{subsec:Comparison-with-subsequent}Comparison with subsequent
developments }

After the preprint version on ArXiv of the present paper had appeared
several interesting new developments have emerged. In \cite[Theorem 2.3]{m-d-c}
the stability result for optimal transport maps in \cite[Cor 3.4]{gi}
(discussed above) was sharpend by replacing $W_{2}$ with $W_{1},$
thus providing an analog of Theorem \ref{thm:(Quantitative-stability--first intro}
when variations of the target measure are considered instead. The
proof of \cite[Theorem 2.3]{m-d-c} is based on an inequality of the
form \ref{eq:analytic ineq reg intro}, but with $\phi_{0}$ and $\phi_{1}$
replaced by their Legendre transforms $\phi_{0}^{*}$ and $\phi_{1}^{*}$
in the left hand side (the constant also depends on a strict lower
bound on $\nabla^{2}\phi,$ i.e. an upper bound on $\nabla^{2}\phi^{*};$
see \cite[Lemma 2.4]{m-d-c}). The latter inequality is shown using
an elegant convexity argument (see also \cite{a-h-a} for a different
Riemannian generalization of \cite[Cor 3.4]{gi}). Moreover, for general
$\mu,$ but with the density of $\nu$ assumed constant, an analog
of the first inequality in Theorem \ref{thm:(Quantitative-stability--second intro}
is established, when variations of the target measure are considered;
see \cite[Theorem 3.1]{m-d-c}. The power of $W_{1}$ obtained in
\cite[Theorem 3.1]{m-d-c} is equal to $2/15.$ Remarkly, the power
is thus independent of the dimension $n$ (which is important for
the applications to machine learning considered in \cite{m-d-c}).
Moreover, an analog of the second inequality in Theorem \ref{thm:(Quantitative-stability--second intro}
is obtained with the $L^{2}-$norm in the right hand side replaced
by an $L^{1}-$norm (i.e. the total variation distance) raised to
the power $1/5.$ One is thus naturally led to ask whether there is
a relation between the quantitative stability with respect to variations
of the target and the source, respectively? The answer is affirmative,
as follows from arguments in the appendix of \cite{m-d-c}. However,
when passing from one type of the inequalites the argument unfortunately
dimishes the exponent of $W_{1};$ it gets divided by $(n+2)$ (hence,
the exponent in Theorem \ref{thm:(Quantitative-stability--second intro}
is larger than the one implied by \cite[Theorem 3.1]{m-d-c}, as long
as $n\leq5).$ This is briefly explained in Section \ref{subsec:Variations-with-respect}. 

In the very recent work \cite{l-n} a variant of the $H^{1}-$estimate
in Theorem \ref{thm:thm H1 intro} is obtained (see \cite[Cor 4.2, formula 4.8]{l-n}).
The main difference between the $H^{1}-$estimate in Theorem \ref{thm:thm H1 intro}
and the one in \cite{l-n} is that in \cite{l-n} the values in $Y$
of $\nabla\phi_{h},$ which are constant on the facets of the weighted
Delaunay tesselation of $X,$ induced by $\phi_{h},$ are replaced
by the barycenters of the the facets $\partial\phi_{h}(x_{i})$ of
the dual tesselation of $Y.$ Moreover, while the constant in the
estimate in Theorem \ref{thm:thm H1 intro} depends on a strictly
lower bound on $\nabla^{2}\phi$ the constant in \cite{l-n} depends
on an \emph{upper} bound of $\nabla^{2}\phi$. The proof in \cite{l-n}
is based on a generalization of the quantitative stability result
for the optimal transport map $T_{\nu}^{\mu}$ in \cite[Cor 3.4]{gi}
(discussed above) to optimal transport plans. The results in \cite{l-n}
also hold for fully-discrete approximation schemes (see \cite[Thm 5.1]{l-n}).
See the end of Section 4 in \cite{l-n} for a detailed comparison
between Theorem \ref{thm:thm H1 intro} (and Theorem \ref{thm:computional geometr})
and the results in \cite{l-n}.

\subsection{Organization}

We start in Section \ref{sec:Preliminaries} with preliminaries from
convex and complex analysis. Since the complex analytic side may not
be familiar to some readers a rather thorough presentation is provided.
In Section \ref{sec:Proof-of-Theorems} the theorems stated in the
introduction are proven, starting with the special case when the density
$g$ is constant on $Y$ and finally reducing to the special case.
In Section \ref{subsec:Relations-to-quantative} the relations to
quantitative stability of optimal transport maps are spelled out.
The relations to computational geometry are explained in Section \ref{thm:computional geometr},
based on Prop \ref{prop:structure of sol over X h}. In the final
section the periodic setting is considered.

\subsection{Acknowledgments}

I am grateful to Quantin Mérigot for illuminating comments and, in
particular, for drawing my attention to the paper \cite{gi}. Thanks
also to the referees whose comments helped to improve the exposition.
This work was supported by grants from the KAW foundation, the Swedish
Research Council and the Göran Gustafsson foundation.

\section{\label{sec:Preliminaries}Preliminaries}

\subsection{\label{subsec:Recap-of-the}Convex analytic notions}

Given a convex function $\phi$ on $\R^{n}$ taking values in $]-\infty,\infty]$
(and not identically equal to $\infty)$ we denote by $\partial\phi$
its\emph{ subgradient, }i.e. the set-valued function on $\R^{n}$
defined by
\[
\partial\phi(x_{0}):=\left\{ y_{0}\in\R^{n}:\,\phi(x_{0})+\left\langle y_{0},x-x_{0}\right\rangle \leq\phi(x)\,\forall x\in\R^{n}\right\} 
\]
(in particular, if $\phi(x_{0})=\infty,$ then $\partial\phi(x_{0})=\emptyset).$
The gradient $(\nabla\phi)(x)$ of a convex function exists a.e. on
the set $\{\phi<\infty\}$ and defines a $L_{loc}^{\infty}-$map into
$\R^{n}$ (called the Brenier map). The \emph{Monge-Ampère measure
}of $\phi$ (in the sense of Alexandrov) is defined by 
\begin{equation}
\int_{B}MA(\phi):=\int_{\partial\phi(B)}dy,\label{eq:MA phi on B}
\end{equation}
for any Borel subset $B$ of $\R^{n}.$ This yields a well-defined
measure on $\R^{n}.$ Indeed, introducing the Legendre transform of
$\phi(x),$ i.e. the convex function $\phi^{*}$ defined by 
\[
\phi^{*}(y):=\sup_{x\in\R^{n}}\left(y\cdot x-\phi(x)\right),
\]
the following formula holds:
\begin{equation}
MA(\phi)=(\nabla\phi^{*})_{*}dy,\label{eq:MA phi as push}
\end{equation}
 i.e. $MA(\phi)$ is the measure obtained as the push-forward of the
Lebesgue measure $dy$ under the $L_{loc}^{\infty}-$map $\nabla\phi^{*}$(the
formula follows from point 2 and 3 below). 

We recall the following basic properties which hold for a given lower
semi-continuous (lsc) convex function $\phi:\R^{n}\rightarrow]-\infty,\infty]$:
\begin{enumerate}
\item $\phi^{**}=\phi$
\item $y\in\partial\phi(x)\iff x\in\partial\phi^{*}(y)$
\item $\overline{(\partial\phi)(\R^{n})}=\overline{\{\phi^{*}<\infty\}}$
(which is a convex set)
\item $\overline{(\partial\phi)(\R^{n})}=\overline{(\partial\phi)(\text{supp\ensuremath{(MA(\phi)}})}$ 
\item If $y\in(\partial\phi)(\R^{n}),$ i.e. $y\in\partial\phi(x_{y})$
for some $x_{y}\in\R^{n},$ then 
\[
\phi^{*}(y)=\sup_{x\in(\partial\phi)(\R^{n})}x\cdot y-\phi(x)=x_{y}\cdot y-\phi(x_{y})
\]
\item The function $\phi$ is \emph{piecewise affine on $\R^{n}$} if and
only if it is the max of a finite number of affine functions, i.e.
there exists $y_{1},...,y_{M}\in\R^{n}$ (assumed distinct) and $c_{1},...,c_{M}$
in $\R$ such that
\[
\phi(x)=\max_{\{y_{i}\}}x\cdot y_{i}-c_{i}
\]
\item Denote by $F_{i}$ the closure of a maximal open region where the
piecewise affine function $\phi$ is affine. Then we can label $F_{i}$
such that $\nabla\phi(x)=y_{i}$ on the interior of $F_{i}.$ Moreover,
the corresponding covering of $\R^{n}$ defines a polyhedral cell-complex
with facets $F_{i}.$
\item Denote by $x_{i},...,x_{N}$ the $0-$dimensional cells (i.e. vertices)
of the polyhedral cell-complex above. Then 
\[
\text{supp\ensuremath{(MA(\phi)}})=\{x_{1},...,x_{N}\}
\]
 and $\partial\phi(x_{i})$ is the convex hull of the vectors $y_{j}$
associated to all facets $F_{j}$ containing $x_{i}.$
\end{enumerate}
A reference for point 1-5 is \cite{v1} and 6 could be taken as the
definition of a convex piecewise affine function and then point 7
follows readily. As for point 8 it can be shown using the following
observation: $x$ is not in a $0-$dimensional cell if and only if
the convex hull $C_{x}(=\partial\phi(x))$ of the vectors $y_{i}$
corresponding to the facets $F_{i}$ containing $x$ has dimension
$p<n$ (as can be seen by identifying the $F_{i}$s with intersecting
pieces of hyperplanes in the graph of $\phi$ in $\R^{n+1}$ and $C_{x}$
with the corresponding normal cone at $(x,\phi(x)).$ As a consequence,
if $x$ is not in a $0-$dimensional cell, then there exists a whole
neighborhood $U$ of $x$ having the latter property and hence $(\partial\phi)(U)$
is a null-set for Lebesgue measure, i.e. $MA(\phi)=0$ in $U.$ This
shows that the support of $MA(\phi)$ is contained in $\{x_{1},...,x_{n}\}.$
Conversely, if $x$ is contained in the latter set, then $p=0$ and
hence $(\partial\phi)(x)$ has dimension $n,$ i.e. $MA(\phi)\{x\}\neq0,$ 

\subsubsection{\label{subsec:The-class-}The class $\mathcal{C}_{Y}(\R^{n})$ of
convex functions associated to a bounded convex domain $Y$}

Given a bounded convex domain $Y$ we denote by $\mathcal{C}_{Y}(\R^{n})$
the space of all convex functions on $\R^{n}$ such that $(\partial\phi)(\R^{n})\subset\overline{Y}.$
A reference element in $\mathcal{C}_{Y}(\R^{n})$ is provided by the\emph{
support function} of $Y:$
\begin{equation}
\phi_{Y}(x):=\sup_{y\in Y}x\cdot y\label{eq:def of support f}
\end{equation}
A function $\phi$ is in $\mathcal{C}_{Y}(\R^{n})$ if and only if
there exists a constant $C$ such that $\phi\leq\phi_{Y}+C$ (as follows
from properties 3 and 5 in the previous section). A leading role in
the present paper will be played by the subspace $\mathcal{C}_{Y}(\R^{n})_{+}$
of $\mathcal{C}_{Y}(\R^{n})$ consisting of the convex functions $\phi$
on $\R^{n}$ with ``maximal growth'' in the sense that the reversed
version of the previous inequality also holds: 
\[
\mathcal{C}_{Y}(\R^{n})_{+}:=\left\{ \phi\,\text{convex\,on \ensuremath{\R^{n}\,\,\text{and\,}}}\phi-\phi_{Y}\in L^{\infty}(\R^{n})\right\} 
\]
In particular, $\phi_{Y}$ is in $\mathcal{C}_{Y}(\R^{n})_{+}.$ If
$\phi\in\mathcal{C}_{Y}(\R^{n})_{+},$ then $MA(\phi)/V(Y)$ is a
probability measure (by point 3 in the previous section). The converse
is not true in general, but we will have great use the fact that if
$MA(\phi)/V(Y)$ is a probability measure and moreover $MA(\phi)$
has compact support, then $\phi$ is in $\mathcal{C}_{Y}(\R^{n})_{+},$
as follows from the following lemma: 
\begin{lem}
\label{lem:l infty bound in terms of Sob}Assume that $\phi$ is in
$\mathcal{C}_{Y}(\R^{n})$ and $MA(\phi)/V(Y)$ is a probability measure.
If $\phi$ is normalized so that $\sup_{\R^{n}}(\phi-\phi_{Y})=0$
(or equivalently, $\inf_{Y}\phi^{*}=0$), then for any $q>n,$
\[
\left\Vert \phi-\phi_{Y}\right\Vert _{L^{\infty}(\R^{n})}\leq\frac{d(Y)}{V(Y)}\int_{\R^{n}}|x|MA(\phi)+C_{n,q}\frac{d(Y)^{\left(1+n(1-1/q)\right)}}{V(Y)}\int_{\R^{n}}|x|^{q}MA(\phi),
\]
where $C_{n,q}$ only depends on $n$ and $q$ and $d(Y)$ and $V(Y)$
denote the diameter and volume of $Y,$ respectively. 
\end{lem}

\begin{proof}
Following the argument in the proof of \cite[Prop 2.2]{be-be} we
denote by $v:=\phi^{*}$ is the Legendre transform of $\phi,$ which
defines a convex function on the interior of $Y.$ By the Sobolev
inequality for the embedding $W^{1,q}(Y)\Subset L^{\infty}(Y)$ we
have, since the interior of $Y$ is a bounded convex domain,
\[
\sup_{Y}|v|\leq\frac{1}{V(Y)}\int_{Y}|v(y)|dy+C_{n,q}\frac{\text{diam}(Y)^{\left(1+n(1-1/q)\right)}}{V(Y)}\left(\int_{Y}|\nabla v(y)|^{q}dy\right)^{1/q}
\]
(see \cite[Lemma 1.7.3]{da} or \cite[Thm 4.4]{m-t-s-o}). In general,
as explained in \cite[Prop 2.2]{be-be}, $\inf_{Y}v=-\sup(\phi-\phi_{Y})$
and hence, by assumption, $\inf_{P}v=0.$ Assume that the infimum
is attained at $y_{0}\in Y,$ i.e. $v(y_{0})=0.$ By convexity $|v(y)|=v(y)-v(y_{0})\leq\nabla v(y)\cdot(y-y_{0}).$
Thus the Cauchy-Schwartz inequality yields 
\[
\text{\ensuremath{\int_{Y}|v(y)|dy\leq d(Y)}}\int_{Y}|\nabla v|dy.
\]
Finally, the proof is concluded by observing that $\int|\nabla v(y)|^{\alpha}dy=\int|x|^{\alpha}MA(\phi)$
for any $\alpha>0$ and
\begin{equation}
\left\Vert \phi-\phi_{Y}\right\Vert _{L^{\infty}(\R^{n})}=\left\Vert \phi^{*}\right\Vert _{L^{\infty}(Y)}\label{eq:Legendre isometry}
\end{equation}
 (which follows directly from the fact that the transformation $\phi\mapsto\phi^{*}$
is decreasing and involutive).
\end{proof}

\subsubsection{\label{subsec:The-2nd-boundary}The 2nd boundary value problem for
the Monge-Ampère equation}

We next recall some basic properties of the second boundary value
problem for the Monge-Ampère equation, introduced in Section \ref{subsec:Background intro}.
We thus let $X$ and $Y$ be bounded domains in $\R^{n},$ with $Y$
bounded and convex. Given a function $g$ with support $Y$ such that
$g\in L^{1}(Y,dy)$ the corresponding ``$g-$Monge-Ampère measure''
$MA_{g}(\phi)$ is defined by 
\[
\int_{B}MA(\phi):=\int_{\partial\phi(B)}gdy,
\]
 (see \cite[Section 3]{m-t-w} for a more general setting involving
a cost function $c).$ In particular, if $\phi$ is smooth than $MA_{g}(\phi)$
is given by formula \ref{eq:formul for MA g intro}. In the case when
$gdy$ has unit integral we denote by $\nu$ the corresponding probability
measure: 
\[
\nu:=gdy
\]

\begin{lem}
\label{lem:uniqueness and extension}Given a probability measure $\mu$
with compact support contained in $X,$ a solution $\phi_{X}$ to
the corresponding second boundary value problem \ref{eq:ma eq with mu intro},
\ref{eq:sec bd cond intro} exists and is uniquely determined (mod
$\R).$ Moreover, $\phi_{X}$ is equal to the restriction to $X$
of $\phi_{\R^{n}}$ (mod $\R).$ 
\end{lem}

\begin{proof}
This result goes back to Alexandrov's classical work on Monge-Ampère
equations, but for the convenience of the reader we show here how
to deduce the result from Brenier's theorem \cite{br}. Given a closed
subset $F$ in $\R^{n}$ denote by $\chi_{F}$ the functions which
is equal to $0$ on $F$ and $\infty$ on the complement of $F.$
In the proof we will refer to point $1-8$ in Section \ref{subsec:Recap-of-the}.
Set $\psi_{X}:=(\chi_{\overline{X}}+\phi_{X})^{*}$ (where we have
used that $\phi_{X}$ extends uniquely to a Lipschitz continuous function
on the closure of $X).$ Since $\mu$ and $\nu$ have the same mass
the second boundary condition \ref{eq:sec bd cond intro} implies
that $\overline{\partial\phi_{X}(\R^{n})}=\bar{Y.}$ Hence $\psi_{X}$
is finite on $Y$ (by point 3) and the MA-equation \ref{eq:ma eq with mu intro}
is equivalent to 
\begin{equation}
\mu=(\nabla\psi_{X})_{*}\nu.\label{eq:transport eq in pf lemma}
\end{equation}
By Brenier's uniqueness theorem \cite{br} the latter equation determines
the $L^{\infty}-$map $\nabla\psi_{X}$ a.e $\nu.$ Since the support
$Y$ of $\nu$ is connected the restriction $\psi$ of $\psi_{X}$
to $Y$ is thus uniquely determined (mod $\R).$ Now, by point 1 $\chi_{\overline{X}}+\phi_{X}=\psi_{X}^{*}$
and since $\overline{\partial(\chi_{\overline{X}}+\phi_{X})(\R^{n})}\Subset\bar{Y}$
we get $\chi_{\overline{X}}+\phi_{X}=(\chi_{\overline{Y}}+\psi_{X})^{*}$
(by point 5). In particular, on $X$ we have $\phi_{X}=(\chi_{\overline{Y}}+\psi)^{*},$
where $\psi,$ as explained above, is independent of $X.$ Hence,
replacing $X$ with $\R^{n}$ reveals that the corresponding two solutions
coincide on $X$ (mod $\R)$ as desired. Also note that Brenier's
existence theorem says that, given $\mu$ and $\nu,$ there exists
some convex function $\psi$ on $Y$ satisfying the equation \ref{eq:transport eq in pf lemma}.
Hence defining $\phi_{X}$ in terms of $\psi$ as above yields the
existence of a solution to the second boundary value problem. 
\end{proof}
The uniqueness property in the previous lemma implies the follow qualitative
stability result:
\begin{prop}
\label{prop:qual stab}(Qualitative stability) Using the setup in
the previous lemma let $\phi_{h}$ be a solution to the 2nd boundary
value problem obtained by replacing $\mu$ with $\mu_{h}$ where $\mu_{h}$
is a family of probability measures on $X$ converging weakly towards
$\mu$ as $h\rightarrow0.$ Then $\nabla\phi_{h}$ converges weakly
towards $\nabla\phi_{X}$ (more precisely, all the distributional
derivatives of $\phi_{h}$ converge weakly towards the distributional
derivatives of $\phi_{X}).$ 
\end{prop}

\begin{proof}
Let $\phi_{h}$ be the convex continuous solution on $\R^{n}$ normalized
by $\phi_{h}(x_{0})=0$ for a fixed point $x_{0}\in X.$ Since $\partial\phi_{h}\subset Y$
and $Y$ is assumed bounded it follows from the Arzela-Ascoli compactness
theorem that there exists a subsequence $\phi_{h_{j}}$ converging
uniformly to a convex function $\phi_{\infty}$ on $X$ as $j\rightarrow\infty.$
By standard continuity properties of Monge-Ampère operators it follows
that $MA_{\nu}(\phi_{\infty})=\mu$ (see \cite[Corollary 3.1]{m-t-w}
for a more general result). But then we can apply the uniqueness property
in the previous theorem to conclude that $\phi_{\infty}=\phi,$ where
$\phi$ is the unique solution to \ref{eq:ma eq with mu intro}, \ref{eq:sec bd cond intro}
on $\R^{n}$ satisfying $\phi(x_{0})=0$ and hence the whole family
$\phi_{h}$ converges locally uniformly on $\R^{n}$ towards $\phi.$
The last statement then follows directly from the definition of distributional
derivatives.
\end{proof}

\subsection{\label{subsec:Recap-of-complex}Complex analytic notions}

For the benefit of the reader lacking background in complex analysis
and geometry we provide a (hopefully user-friendly) recap of some
complex analytic notions (see the book \cite{de} for further general
background and \cite{be-be} for the case of toric varieties). Setting
$z:=x+iy\in\C^{n}$ the space $\Omega^{1}(\C^{n})$ of all complex
one-forms on $\C^{n}$ decomposes as a sum
\begin{equation}
\Omega^{1}(\C^{n})=\Omega^{1,0}(\C^{n})+\Omega^{0,1}(\C^{n}),\label{eq:decomp of Omega one}
\end{equation}
of the two subspaces spanned by $\{dz_{i}\}$ and $\{d\bar{z}_{i}\},$
respectively. This induces a decomposition of the exterior algebra
of all complex differential forms $\Omega^{\cdot}(\C^{n})$ into forms
of bidegree $(p,q),$ where $p\leq n$ and $q\leq n.$ Accordingly,
the exterior derivative $d$ decomposes as $d=\partial+\bar{\partial},$
where 

\[
\partial\phi:=\sum_{i=1}^{n}\frac{\partial\phi}{\partial z_{i}}dz_{i},\,\,\,\frac{\partial}{\partial z_{i}}:=(\frac{\partial}{\partial x_{i}}-i\frac{\partial}{\partial y_{i}})/2,
\]
 and taking its complex conjugate defines the $(0,1)-$form $\bar{\partial}\phi.$
In particular, 
\begin{equation}
\omega^{\phi}:=\frac{i}{2\pi}\partial\bar{\partial}\phi=\frac{i}{2\pi}\sum_{i,j\leq n}\frac{\partial^{2}\phi}{\partial z_{i}\partial\bar{z}_{j}}dz_{i}\wedge d\bar{z}_{j},\label{eq:omega phi}
\end{equation}
 defines a real $(1,1)-$form (this normalization turns out to be
useful, as illustrated by Example \ref{ex:normalization in C} \ref{ex:normalization in C}
below). Such a smooth form $\omega$ is said to be positive (Kähler),
written as $\omega\geq0$ $(\omega>0),$ if the corresponding Hermitian
matrix is semi-positive (positive definite) at any point. Positivity
can also be defined for general $(p,p)-$forms, but for the purpose
of the present paper it is enough to know that 
\[
\omega_{i}\geq0,\,\,\,\,i=1,..,n\implies\omega_{1}\wedge\cdots\wedge\omega_{n}\geq0,
\]
 where the last inequality holds in the sense of measures.

If $\phi$ is smooth, then $\phi$ is said to be \emph{plurisubharmonic
(psh)} if $\omega^{\phi}\geq0.$ A general function $\phi\in L_{loc}^{1}$
is said to be psh if it is strongly upper semi-continuous and $\omega^{\phi}\geq0$
holds in the weak sense of currents. 
\begin{example}
\label{ex:normalization in C}The normalization used in the definition
of $\omega^{\phi}$ ensures that when $n=1$ the measure $\omega^{\phi}$
on $\C$ is the Dirac measure at $0\in\C$ in the case when $\phi=\log|z|^{2}.$
\end{example}

More generally, given a real $(1,1)-$form $\omega_{0}$ a function
$u$ said to be\emph{ $\omega_{0}-$psh }if 
\[
\omega_{u}:=\omega_{0}+\frac{i}{2\pi}\partial\bar{\partial}u\geq0
\]
If $\omega_{0}:=\omega^{\phi_{0}}$ this means that $\phi$ is psh
if and only if $u:=\phi-\phi_{0}$ is $\omega_{0}-$psh. When $\phi_{i}$
for $i=1,...,p$ is psh and in $L_{loc}^{\infty}$ the positive closed
$(p,p)-$currents 
\[
\omega^{\phi_{1}}\wedge\cdots\wedge\omega^{\phi_{p}},
\]
is defined by the local pluripotential theory of Bedford-Taylor. The
current does not charge pluripolar subsets (i.e sets locally contained
in the $-\infty-$locus of a psh function) and in particular not analytic
subvarieties. Accordingly, the Monge-Ampère measure of a locally bounded
psh function $\phi(z),$ is defined by the measure
\begin{equation}
MA_{\C}(\phi):=(\omega^{\phi})^{n}/n!.\label{eq:def of complex MA}
\end{equation}
We also recall that, if $\phi_{0}$ and $\phi_{1}$ are as above then
the positive current $i\partial(\phi_{0}-\phi_{1})\wedge\bar{\partial}(\phi_{0}-\phi_{1})$
may be defined by the formula 
\[
\partial\varphi\wedge\bar{\partial}\varphi:=-\varphi\partial\bar{\partial}\varphi+\partial\bar{\partial}\varphi^{2},\,\,\,\varphi:=\phi_{0}-\phi_{1}
\]

\subsubsection{Metrics on line bundles and $\omega_{0}-$psh functions}

The local complex analytic notions above naturally extend to the global
setting where $\C^{n}$ is replaced by a complex manifold (since the
decomposition \ref{eq:decomp of Omega one} is invariant under a holomorphic
change of coordinates). However, if $X$ is compact, then all psh
functions $\phi$ on $X$ are constant (by the maximum principle).
Instead the role of a (say, smooth) psh function $\phi$ on $\C^{n}$
is played by a positively curved metric on a line bundle $L\rightarrow X$
(using additive notations for metrics). To briefly explain this first
recall that a line bundle $L$ over a complex manifold $X$ is, by
definition, a complex manifold (called the total space of $L)$ with
a surjective holomorphic map $\pi$ to $X$ such that the fibers $L_{x}$
of $\pi$ are one-dimensional complex vector spaces and such that
$\pi$ is locally trivial. In other words, any point $x\in X$ admits
a neighborhood $U$ such that $\pi:\,L\rightarrow U$ is (equivariantly)
isomorphic to the trivial projection $U\times\C\rightarrow U.$ Fixing
such an isomorphism, holomorphic sections of $L\rightarrow U$ may
be identified with holomorphic functions on $U.$ In particular, the
function $1$ over $U$ corresponds to a non-vanishing holomorphic
section $s_{U}$ of $L\rightarrow U.$ Now, a smooth (Hermitian) metric
$\left\Vert \cdot\right\Vert $ on the line bundle $L$ is, by definition,
a smooth family of Hermitian metrics on the one-dimensional complex
subspaces $L_{x},$ i.e. a one-homogeneous function on the total space
of the dual line bundle $L^{*}$ which vanishes precisely on the zero-section.
Given a covering $U_{i}$ of $X$ and trivializations of $L\rightarrow U_{i}$
a metric $\left\Vert \cdot\right\Vert $ on $L$ may be represented
by the following family of local functions $\phi_{U_{i}}$ on $U_{i}:$
\[
\phi_{U_{i}}:=-\log\left\Vert s_{U_{i}}\right\Vert ^{2}
\]
(accordingly a metric on $L$ is often, in additive notation, denoted
by the symbol $\phi).$ Even if the functions $\phi_{U_{i}}$ do not
agree on overlaps, the (normalized) curvature form $\omega$ of the
metric $\left\Vert \cdot\right\Vert $ is a well-defined closed two-form
on $X,$ locally defined by 
\[
\omega_{|U_{i}}:=\omega^{\phi_{U_{i}}}
\]
Singular metrics on $L$ may be defined in a similar way. In particular,
a singular metrics is said to have \emph{positive curvature} if the
local functions $\phi_{U_{i}}$ are psh, i.e. the corresponding curvature
form $\omega$ defines a positive $(1,1)-$current on $X.$ The difference
of two metrics, written as $\phi_{1}-\phi_{2}$ in additive notation,
is always a globally well-defined function on $X$ (as a consequence,
the curvature currents of any two metrics on $L$ are cohomologous
and represent the first Chern class $c_{1}(L)\in H^{2}(X,\Z)).$ Fixing
a reference metric $\phi_{0}$ and setting $u:=\phi-\phi_{0}$ this
means that the space all metrics $\phi$ on $L$ with positive curvature
current may be identified with the space of all $\omega_{0}-$psh
functions $u$ on $X.$
\begin{example}
\label{exa:proj space and varieties}The $m-$dimensional \emph{complex
projective space} $\P^{m}:=\C^{m+1}-\{0\}/\C^{*})$ comes with a tautological
line bundle whose total space is $\C^{m+1}$ and the line over a point
$[Z_{0}:....:Z_{m}]\in\P^{m}$ is the line $\C(Z_{0},..,Z_{m}).$
The dual of the tautological line bundle is called the \emph{hyperplane
line bundle }and is usually denoted by $\mathcal{O}(1)$ (the notation
reflects the fact that the metrics on $\mathcal{O}(1)$ may be identified
with $1-$homogeneous functions on $\C^{m+1}).$ The Euclidean metric
on $\C^{m+1}$ induces a metric on the tautological line bundle and
hence on its dual $\mathcal{O}(1),$ called the \emph{Fubini-Study
metric.} In the standard affine chart $U:=\C^{n}\subset\P^{n}$ defined
by all points where $z_{0}\neq0$ and with the standard trivializing
section $s_{U}:=(1,z_{1},...,z_{n})$ of the tautological line bundle
the Fubini-Study metric is represented by 
\[
\phi_{FS}(z):=\log\left\Vert s_{U}\right\Vert ^{2}=\log(1+|z|^{2}),
\]
 defining a smooth metric with strictly positive curvature form. Another
$T^{n}-$invariant metric on $\mathcal{O}(1)$ with positive curvature
current, which is continuous, but not smooth, is the ``max-metric''
defined by the one-homogeneous psh function $\max\{|Z_{0}|,...,|Z_{n}|\},$
which may be represented by 
\[
\phi_{max}(z)=\log\max\{1,|z_{1}|^{2},...,|z_{n}|^{2}\},
\]
 Given any complex subvariety $X\Subset\P^{m}$ one obtains a line
bundle $L$ over $X$ by restricting $\mathcal{O}(1)$ to $X$ (and
a smooth positively curved metric $\phi$ by restricting the Fubini-Study
metric). If $X$ is singular then, by Hironaka's theorem it admits
a smooth resolution, i.e. a smooth compact manifold $X'$ with surjective
and generically one-to-one projection $\pi'$ to $X.$ Pulling back
$L$ by $\pi'$ yields a line bundle $L'$ over $X'$ (endowed with
a smooth metric $\phi').$ 
\end{example}

\subsubsection{\label{subsec:The-toric-variety}The toric variety associated to
a moment polytope $Y$ }

Let $Y$ be a bounded closed convex polytope with non-empty interior
and rational vertices. It determines a compact toric complex analytic
variety $X_{\C}$ and an ample line bundle $L$ over $X_{\C}.$ More
precisely, this is the case if the rational polytope $Y$ is replaced
by the integer polytope $kY$ for a sufficiently large positive integer
$k.$ But since scalings of $Y$ will be harmless we may as well assume
that $k=1.$ Then $X_{\C}$ may be defined as the closure of the image
of the holomorphic (algebraic) embedding defined by 
\begin{equation}
\C^{*n}\rightarrow\P^{M-1}\,\,\,\,z\mapsto[z^{m_{0}}:\cdots:z^{m_{M-1}}],\label{eq:toric emb}
\end{equation}
 using multinomial notation, where $m_{0},...,m_{M-1}$ label the
integer vectors in $Y.$ The line bundle $L$ is then simply defined
as the restriction to $X_{\C}$ of the hyperplane line bundle $\mathcal{O}(1)$
on $\P^{M-1}.$ By construction, the following holds:
\begin{itemize}
\item $X_{\C}$ may be embedded in a complex projective space $\P^{N}$
in such a way that $L$ coincides with the restriction to $X_{\C}$
of the hyperplane line bundle $\mathcal{O}(1)$ on $\P^{N}.$ 
\item The standard action of the real $n-$dimensional torus $T^{n}$ on
$\C^{*n}$ extends to a holomorphic action of $X_{\C}$ which lifts
to the line bundle $L.$
\end{itemize}
By the first point above we can identify $\C^{*n}$ with an open dense
subset of $X_{\C},$ whose complement in $X_{\C}$ is an analytic
subvariety. 

Now, the key point is that the function $\phi(x)$ is in the class
$\mathcal{C}_{Y}(\R^{n})$ (Section \ref{subsec:The-class-}) if and
only if $\phi(x)$ extends to a positively curved (singular) metric
on $L\rightarrow X_{\C}.$ More precisely, we have the following \cite[ Prop 3.2]{be-be}
\begin{lem}
\label{lem:conv function as metric}A convex function $\phi(x)$ such
that $\mathcal{C}_{Y}(\R^{n})_{+}$ may be identified with a $L^{\infty}-$metric
on the line bundle $L\rightarrow X_{\C}$ with positive curvature
current $\omega^{\phi}.$
\end{lem}

\begin{proof}
Since the lemma will play a key role in the proof of the main results
we recall, for the benefit of the reader, the simple proof. Without
loss of generality we can assume that $m_{0}:=0$ is in $Y.$ By construction
the ``max metric'' on $\mathcal{O}(1)\rightarrow\P^{M-1}$ restricts
to a continuous (and in particularly bounded) positively curved metric
on $L\rightarrow X_{\C}.$ The map \ref{eq:toric emb} may be factored
as $\C^{*n}\rightarrow\C^{N}\rightarrow\P^{M-1}.$ Hence, the standard
trivialization of $\mathcal{O}(1)$ over the affine piece $\C^{N}$
pulls back to give a trivialization of $L$ over $\C^{*n},$ where
the restricted max metric is represented by $\phi_{Y}(x),$ when expressed
in the logarithmic coordinates $x$ on $\R^{n}.$ Now, any other $L^{\infty}-$metric
$\phi$ on $L\rightarrow X_{\C}$ satisfies $\phi-\phi_{Y}\in L^{\infty}(X_{\C})$
and, as a consequence, restricting to $\C^{*n}$ and switching to
the real logarithmic coordinate $x$ shows that $\phi(x)$ is $\mathcal{C}_{Y}(\R^{n})_{+}.$
Conversely, given $\phi(x)$ in $\mathcal{C}_{Y}(\R^{n})_{+}$ we
have that $\phi(z)$ is psh on $\C^{*n}$ and $u:=\phi(z)-\phi_{Y}(z)$
is in $L^{\infty}(\C^{*n}).$ Since $\C^{*n}$ is dense in $X_{\C}$
and $u\leq C$ it admits a canonical upper semi-continuous (usc) extension
to all of $X_{\C}$ (namely, the smallest one). Since $\phi_{Y}(z)$
extends to define a continuous metric on $L\rightarrow X_{\C}$ this
means that $\phi$ extends from $\C^{*n}$ to a canonical usc metric
on $L\rightarrow X_{\C}$ (in additive notation) which has a positive
curvature current on $\C^{*n}\Subset X_{\C}.$ But the complement
$X_{\C}-\C^{*n}$ is an analytic subvariety of $X_{\C}$ and hence
it follows from basic local extension properties of psh functions
that the corresponding metric on $L\rightarrow X_{\C}$ has positive
curvature current.
\end{proof}
We note that for any $\phi$ as above
\begin{equation}
\frac{1}{n!}\int_{X_{\C}}(\omega^{\phi_{1}})^{n}==V(Y),\label{eq:snitt blir volym}
\end{equation}
 the volume of $Y.$ Indeed, since $MA_{\C}(\phi)$ does not charge
analytic subvarieties the integral can be restricted to $\C^{*n}$
and then formula \ref{eq:MA c and MA} below can be invoked. 

\subsection{\label{subsec:Complex-vs-real}Complex vs real notions}

If $\phi(z)=\phi(x),$ i.e. $\phi$ is independent of the $y-$variable,
then $\phi(z)$ is psh on $\C^{n}$ if and only if $\phi(x)$ is convex
on $\R^{n},$ as follows directly from the relation 
\begin{equation}
\frac{\partial^{2}\phi}{\partial z_{i}\partial\bar{z}_{j}}=\frac{1}{4}\sum_{i,j\leq n}\frac{\partial^{2}}{\partial x_{i}\partial x_{j}}\label{eq:cplex hessian as real}
\end{equation}
For example, if $\phi_{0}(x):=4\pi|x|^{2}/2,$ then $u(x)$ is quasi-convex
(i.e. $u(x):=\phi(x)-\phi_{0}(x)$ is convex) if and only if $u(z)$
is $\omega_{0}-$psh for 
\begin{equation}
\omega_{0}:=\frac{i}{2\pi}\partial\bar{\partial}\phi_{0}=\sum_{i}\frac{i}{2}dz_{i}\wedge d\bar{z}_{i}=\sum_{i}dx_{i}\wedge dy_{i},\label{eq:def of omega noll}
\end{equation}
the standard Kähler form on $\C^{n}.$ The form $\omega_{0}$ descends
to the Abelian variety $\C^{n}/(\Z+i\Z).$

\subsubsection{\label{subsec:The-complex-torus}The complex torus $\C^{*n}$}

Let $\mbox{Log }$be the map from $\C^{*n}$ to $\R^{n}$ defined
by 
\[
\mbox{Log}(z):=x:=(\log(|z_{1}|^{2}),...,\log(|z_{n}|^{2})).
\]
The real torus $T^{n}$ acts transitively on the fibers of the map
Log. Pulling back a convex function $\phi(x)$ on $\R^{n}$ by Log
yields a $T^{n}-$invariant function on $\C^{*n}$ that we will, abusing
notation slightly, denote by $\phi(z).$ The function $\phi(x)$ is
convex if and only if the corresponding function $\phi(z)$ on $\C^{*n}$
is plurisubharmonic. This can be seen by identifying $\C^{*n}$ with
$\C^{n}/i\pi\Z^{n}(=\R^{n}+i\pi T^{n})$ using the multivalued logarithmic
holomorphic coordinate $2\log z,$ and proceeding as in Section \ref{subsec:Complex-vs-real}. 

We will have great use for the following lemma: 
\begin{lem}
\label{lem:complex to real}Let $\phi(x)$ be a finite convex function
on $\R^{n}$ and denote by $\phi(z)$ the corresponding $T^{n}-$invariant
psh function on $\C^{*n}$ defined as the pull-back of $\phi(x)$
under the Log map. Similarly if $u$ is a difference of two finite
convex functions on $\R^{n}$ we denote by $u(z)$ the corresponding
function on $\C^{*n}.$ Then following two identities of measures
on $\R^{n}$ hold:
\begin{equation}
\mbox{(Log)}_{*}\left(\frac{1}{n!}(\frac{i}{2\pi}\partial\bar{\partial}\phi)^{n}\right)=MA(\phi).\label{eq:MA c and MA}
\end{equation}
and 
\[
\mbox{(Log)}_{*}\left(\frac{i}{2\pi}\partial u\wedge\bar{\partial}u\wedge\frac{1}{(n-1)!}(\frac{i}{2\pi}\partial\bar{\partial}\phi)^{n-1}\right)=\left|\nabla u\right|_{g^{\phi}}^{2}MA(\phi),
\]
if $\phi(x)$ is assumed to be smooth and strictly convex, where $g^{\phi}$
denotes the Riemannian metric on $\R^{n}$ defined by the symmetric
Hessian matrix $\nabla^{2}\phi$ and $\nabla u$ is defined almost
everywhere with respect to $dx.$ 
\end{lem}

\begin{proof}
These formulas are essentially well-known (in particular, the first
one), but for completeness a proof is provided. First consider two
smooth functions $\phi(z)$ and $u(z)$ defined on $\C^{n}$ with
holomorphic coordinate $z.$ Assume that $\phi(x+iy)$ and $u(x+iy)$
are independent of $y.$ Then 
\begin{equation}
\frac{1}{n!}(\frac{i}{2}\partial\bar{\partial}\phi)^{n}=\frac{1}{4^{n}}MA(\phi)\wedge dy\label{eq:pf of lemma complex real first formula}
\end{equation}
and

\begin{equation}
\frac{i}{2}\partial u\wedge\bar{\partial}u\wedge\frac{1}{(n-1)!}(\frac{i}{2}\partial\bar{\partial}\phi)^{n-1}=\frac{1}{4^{n}}\left|\nabla u(x)\right|_{g^{\phi}}^{2}MA(\phi)\wedge dy\label{eq:proof of lemma complex real second formula}
\end{equation}
Indeed, without loss of generality we may assume that $\phi(x)=\sum\lambda_{i}|x_{i}|^{2}$
(since it is enough to prove the formulas at a fixed point, which
may be taken as $x=0$ and since the formulas are invariant under
$z\mapsto Az$ for $A\in SO(n)).$ Then $\frac{i}{2}\partial\bar{\partial}\phi=\sum_{i}\frac{1}{4}\frac{\partial^{2}\phi}{\partial x_{i}\partial x_{i}}dx_{i}\wedge dy_{i}$
and hence formula \ref{eq:pf of lemma complex real first formula}
follows directly. As for formula \ref{eq:proof of lemma complex real second formula}
it follows from noting that the term involving $dz_{i}\wedge d\bar{z}_{i}$
in $\partial u\wedge\bar{\partial}u$ only gives a non-zero contribution
when it encounters the product of all terms in $\partial\bar{\partial}\phi$
not involving the index $i.$ Indeed, this gives
\[
\frac{i}{2}\partial u\wedge\bar{\partial}u\wedge\frac{1}{(n-1)!}(\frac{i}{2}\partial\bar{\partial}\phi)^{n-1}=\sum_{i}\frac{1}{4}\left|\frac{\partial u}{\partial x_{i}}\right|^{2}\frac{1}{\lambda_{i}}\lambda_{1}\lambda_{2}\cdots\lambda_{n}dx\wedge dy
\]
which proves formula \ref{eq:proof of lemma complex real second formula}.
Then, by standard local approximation arguments, formula \ref{eq:pf of lemma complex real first formula}
extends to the case when $\phi$ is non-smooth and formula \ref{eq:proof of lemma complex real second formula}
to the case when $u$ is non-smooth. Finally, consider $\C^{*n}$
and denote by $w_{i}$ its standard holomorphic coordinates. This
means that $z_{i}:=2\log w_{i}$ is multivalued on $\C^{*n}$ and
the real part of $z$ is equal to $\text{Log \ensuremath{(w).}}$
But locally $z$ defines holomorphic coordinates on $\C^{*n}$ to
which formula \ref{eq:pf of lemma complex real first formula} and
formula \ref{eq:proof of lemma complex real second formula} apply.
Moreover, we can identify $dy/(4\pi)^{n}$ with the standard invariant
probability measure on the $T^{n}-$fiber of the Log map from $\C^{*n}$
to $\R^{n}.$ Since the push-forward operation amounts to integration
along the fibers, the formulas in the lemma thus follow from formula
\ref{eq:pf of lemma complex real first formula} and formula \ref{eq:proof of lemma complex real second formula}. 
\end{proof}

\section{\label{sec:Proof-of-Theorems}Proof of Theorems \ref{thm:thm H1 intro},
\ref{thm:(General-case) intro}}

Let $X$ and $Y$ be open bounded domains in $\R^{n},$ with $Y$
convex. Assume given a positive function $g\in L^{1}(Y,dy)$ such
\begin{equation}
\delta:=\inf_{Y}g>0\label{eq:def of delta in text}
\end{equation}
Recall that $MA_{g}(\phi)$ denotes the ``$g-$Monge-Ampère measure''
of a given convex function $\phi,$ defined in Section \ref{subsec:The-2nd-boundary}.
When $g=1$ we simply write $MA_{g}=MA.$ 

\subsection{The key analytic inequalities when $g=1$ }

We first consider when $g=1,$ starting with the case of a uniformly
convex $\phi_{0}.$ 
\begin{prop}
\label{prop:key ineq}Given bounded open domains $X$ and $Y$ with
$Y$ assumed convex, assume that $\phi_{0}$ and $\phi_{1}$ are convex
functions on $X,$ such that the the closures of the sub-gradient
images $(\partial\phi_{i})(X)$ are equal to $\overline{Y}.$ If there
exists a positive constant $C_{0}$ such that $\nabla^{2}\phi_{0}\geq C_{0}^{-1}I$
(in the sense of distributions), then 
\[
\int_{X}|\nabla\phi_{0}-\nabla\phi_{1}|^{2}dx\leq nC_{0}^{n-1}\int_{X}(\phi_{1}-\phi_{0})(MA(\phi_{0})-MA(\phi_{1})).
\]
 
\end{prop}

In order to prove this we start with some preparations. First, by
Lemma \ref{lem:complex to real}
\begin{equation}
\int_{X}|\nabla(\phi_{1}-\phi_{0})|_{g^{\phi}}^{2}MA(\phi)=\int_{(\text{Log})^{-1}(X)}\frac{i}{2\pi}\partial u\wedge\bar{\partial}u\wedge\frac{1}{(n-1)!}(\omega^{\phi})^{n-1}\label{eq:Hessian norm as complex}
\end{equation}
holds for for any smooth convex function $\phi$ on $X.$ Next, we
will use the same notation $\phi_{0}$ and $\phi_{1}$ for the canonical
extensions to $\R^{n}$ solving the corresponding second boundary
problem on all of $\R^{n}$ (as in Lemma \ref{lem:uniqueness and extension}).
In general, if $\phi(x)$ is a convex function on $\R^{n}$ we will
denote by $\phi(z)$ the corresponding $T^{n}-$invariant plurisubharmonic
function on $\C^{*n},$ obtained by pulling back $\phi$ to $\C^{*n}$
using the Log map (abusing notation slightly, as in Section \ref{subsec:The-complex-torus}).
It will be enough to prove the following identity:
\begin{equation}
-\int_{\C^{*n}}(\phi_{0}-\phi_{1})\left((\omega^{\phi_{0}})^{n}-(\omega^{\phi_{1}})^{n}\right)=\frac{i}{2\pi}\sum_{j=0}^{n}\int_{\C^{*n}}\partial(\phi_{0}-\phi_{1})\wedge\bar{\partial}(\phi_{0}-\phi_{1})\wedge(\omega^{\phi_{0}})^{n-j}\wedge(\omega^{\phi_{1}})^{j},\label{eq:exp of I on cplx torus}
\end{equation}
Indeed, first applying formula \ref{eq:Hessian norm as complex} to
$\phi=|x|^{2}/2$ a gives 
\[
\int_{X}|\nabla\phi_{0}-\nabla\phi_{1}|^{2}dx=\frac{i}{2\pi}\int_{(\text{Log})^{-1}(X)}\partial(\phi_{0}-\phi_{1})\wedge\bar{\partial}(\phi_{0}-\phi_{1})\wedge(\omega^{\phi})^{n-1}\leq
\]
\[
\leq C_{0}^{n-1}\frac{i}{2\pi}\int_{(\text{Log})^{-1}(X)}\partial(\phi_{0}-\phi_{1})\wedge\bar{\partial}(\phi_{0}-\phi_{1})\wedge(\omega^{\phi_{0}})^{n-1},
\]
 using that, by assumption, $\omega^{\phi_{0}}\geq C_{0}^{-1}\omega^{\phi}.$
Finally, using that all the integrands in the rhs of formula \ref{eq:exp of I on cplx torus}
are non-negative will then conclude the proof. In order to prove formula
\ref{eq:exp of I on cplx torus} first note that
\[
-\int_{\C^{*n}}(\phi_{0}-\phi_{1})\left((\omega^{\phi_{0}})^{n}-(\omega^{\phi_{1}})^{n}\right)=-\sum_{j=0}^{n}\int_{\C^{*n}}(\phi_{0}-\phi_{1})\wedge\frac{i}{2\pi}\partial\bar{\partial}(\phi_{0}-\phi_{1})\wedge(\omega^{\phi_{0}})^{n-j}\wedge(\omega^{\phi_{1}})^{j},
\]
 as follows directly from the algebraic identity

\begin{equation}
(\omega^{\phi_{1}})^{n}-(\omega^{\phi_{0}})^{n}=\sum_{j=1}^{n-1}(\omega^{\phi_{1}}-\omega^{\phi_{0}})\wedge(\omega^{\phi_{0}})^{n-j}\wedge(\omega^{\phi_{1}})^{j}\label{eq:exp of MA minus MA}
\end{equation}
This means that if integration by parts can be justified, then the
desired identity \ref{eq:exp of I on cplx torus} follows. However,
the non-compactness of $\C^{*n}$ poses non-trivial difficulties,
so instead of working directly on $\C^{*n}$ we will use a compactification
argument, which applies when $Y$ is a rational polytope (the general
case will then follow by approximation).

\subsubsection*{Compactification when $Y$ is a convex polytope with rational vertices}

Let $X_{\C}$ be the compact toric complex analytic variety determined
by the moment polytope $\bar{Y}$ and denote by $L$ the corresponding
ample line bundle over $X_{\C}$ (Section \ref{subsec:The-toric-variety}).
By Lemma \ref{lem:l infty bound in terms of Sob}, combined with Lemma
\ref{lem:conv function as metric}, the $T^{n}-$invariant psh functions
$\phi_{0}(z)$ and $\phi_{1}(z)$ on $\C^{*n}$ extend to define $L^{\infty}-$metrics
on the line bundle $L\rightarrow X_{\C}$ with positive curvature
currents $\omega^{\phi_{0}}$ and $\omega^{\phi_{1}}.$ In particular,
$\phi_{0}-\phi_{1}\in L^{\infty}(X_{\C}).$ We claim that
\begin{equation}
\int_{X_{\C}}(\phi_{1}-\phi_{0})\left((\omega^{\phi_{0}})^{n}-(\omega^{\phi_{1}})^{n}\right)=\sum_{j=0}^{n}\int_{X_{\C}}\frac{i}{2\pi}\partial(\phi_{0}-\phi_{1})\wedge\bar{\partial}(\phi_{0}-\phi_{1})\wedge(\omega^{\phi_{0}})^{n-j}\wedge(\omega^{\phi_{1}})^{j}\label{eq:I on toric var}
\end{equation}
This is a well-known identity in Kähler geometry (in the smooth case)
and global pluripotential theory (in the general singular $L^{\infty}-$case)
and follows from the general integration by parts formula for psh
functions in $L_{loc}^{\infty}$ in \cite[Thm 1.14]{begz} (see \cite[Formula 2.9]{bbgz}).
But for the benefit of the reader we provide the following alternative
proof. First, assume that the metrics $\phi_{i}$ on $L$ are smooth
(i.e. the restrictions to $L\rightarrow X$ of smooth metrics on $\mathcal{O}(1)\rightarrow\P^{N}).$
Expanding point-wise, as in formula \ref{eq:exp of MA minus MA} and
using that the form $\Theta:=(\omega^{\phi_{0}})^{n-j}\wedge(\omega^{\phi_{1}})^{j}$
is closed formula \ref{eq:I on toric var} then follows from Stokes
formula if $X_{\C}$ is non-singular and from Stokes formula on a
non-singular resolution of $X_{\C}$ if $X_{\C}$ is singular. To
handle the general case we invoke the general fact that any (possibly
singular) metric on an ample line bundle $L$ over a projective complex
variety can be written as a decreasing limit of smooth metrics \cite{c-g-z}
(in fact, this can be shown by a simple direct argument in the present
toric setting). Formula \ref{eq:I on toric var} then follows from
the previous smooth case, combined with the continuity of expressions
of the form appearing in formula \ref{eq:I on toric var} under decreasing
limits (see \cite[Prop 2.8]{begz} and its proof for much more general
convergence results). 

\subsubsection*{Conclusion of proof of Proposition \ref{prop:key ineq}}

Let $Y$ be as in the previous step. Since the complex Monge-Ampère
measures of locally bounded psh functions do not charge complex analytic
subvarieties formula \ref{eq:I on toric var} on $\C^{*n}$ follows
from formula \ref{eq:I on toric var} on $X_{\C}.$ Since each term
in the right hand side above is non-negative we deduce that 
\[
-\int_{\C^{*n}}(\phi_{0}-\phi_{1})\left((\omega^{\phi_{0}})^{n}-(\omega^{\phi_{1}})^{n}\right)\geq\frac{1}{n!}\frac{i}{2\pi}\int_{(\text{Log})^{-1}(X)}\partial(\phi_{0}-\phi_{1})\wedge\bar{\partial}(\phi_{0}-\phi_{1})\wedge(\omega^{\phi_{0}})^{n}
\]
By \ref{eq:MA c and MA} and \ref{eq:Hessian norm as complex} this
proves formula \ref{eq:I on toric var} when $Y$ is a convex polytope
with rational vertices. In the general case, we can write $Y$ is
an increasing limit of rational convex polytopes $Y_{j}.$ Then the
general case follows from the previous case and basic qualitative
stability properties for the solution of the second boundary value
problem for $MA$ on $X$ with respect to variations of the target
domain $Y.$ 

\subsubsection{The key inequality for general $\phi_{0}$ and $\phi_{1}$}

We next turn to the key analytic inequality in the case of general
$\phi_{0},$ but still with $g=1.$ 
\begin{prop}
\label{prop:key non reg}Given bounded open domains $X$ and $Y$
with $Y$ assumed convex, assume that $\phi_{0}$ and $\phi_{1}$
are convex functions on $X,$ such that the the closures of the sub-gradient
images $(\partial\phi_{i})(X)$ are equal to $\overline{Y}.$ Then
there exists a constant $C$ only depending on $X$ and $Y$ such
that
\[
\int_{X}|\nabla\phi_{0}-\nabla\phi_{1}|^{2}dx\leq C\left(\int_{X}(\phi_{\text{1}}-\phi_{0})(MA(\phi_{1})-MA(\phi_{0}))\right)^{1/2^{n-1}}
\]
More precisely, the constant $C$ only depends on upper bounds on
the the diameters of $X$ and $Y,$ $d(X)$ and $d(Y)$ and a positive
lower bound on the volume $V(Y)$ of $Y$ 
\end{prop}

\begin{proof}
It will be enough to show that
\begin{equation}
\int_{X}|\nabla\phi_{0}-\nabla\phi_{1}|_{g^{\phi^{Y}}}^{2}MA(\phi^{Y})\leq C_{Y}\left(\int_{X}(\phi_{1}-\phi_{0})(MA(\phi_{0})-MA(\phi_{1}))\right)^{1/2^{n-1}},\label{eq:ineq in pf prop key reg}
\end{equation}
for a fixed $\phi^{Y}\in\mathcal{C}_{Y}(\R^{n})_{+}$ such that 
\[
\nabla^{2}\phi^{Y}\leq A_{1}I,\,\,\,MA(\phi^{Y})\geq A_{2}dx
\]
 for some positive constants $A_{1}$ and $A_{2}.$ In fact, setting
\begin{equation}
R(X):=(\sup_{x\in X}|x|)\label{eq:outer radius}
\end{equation}
we will show that $C_{Y}$ only depends on upper bounds on $R(X),R(Y)$
and 
\begin{equation}
A_{3}:=\left\Vert \phi^{Y}-\phi_{Y}\right\Vert _{L^{\infty}(\R^{n})/\R}:=\inf_{c\in\R}\left\Vert (\phi^{Y}+c)-\phi_{Y}\right\Vert \label{eq:def of A three}
\end{equation}
To this end we may, without loss of generality, assume that $\phi_{0}$
and $\phi_{1}$ are sup-normalized as in Lemma \ref{lem:l infty bound in terms of Sob}.
Fix a sequence of rational convex polytopes $Y_{j}$ increasing to
$Y$ and some (say smooth) functions $\phi^{Y_{j}}\in\mathcal{C}_{Y}(\R^{n})_{+}$
such that 
\[
\phi^{Y_{j}}-\phi_{Y_{j}}\rightarrow\phi^{Y}-\phi_{Y}
\]
 uniformly on $\R^{n}.$ Just as in the proof of Theorem \ref{thm:thm H1 intro}
it will be enough to prove that the inequality in the proposition
holds when $Y$ is a rational convex polytope and $\phi^{Y}$ is replaced
by $\phi^{Y_{j}}$ (by letting $j\rightarrow\infty$ in the end).
We will identify, just as in the proof of Theorem \ref{thm:thm H1 intro}
a convex function $\phi(x)$ on $\R^{n}$ with a psh function $\phi(z)$
on $\C^{*n}.$ Setting $\omega_{Y_{j}}:=\omega^{\phi^{Y_{j}}}$ it
will thus be enough to show that
\[
i\int_{\C^{*n}}\partial(\phi_{0}-\phi_{1})\wedge\bar{\partial}(\phi_{0}-\phi_{1})\wedge\omega_{Y_{j}}^{n-1}\leq C_{Y_{j}}\left(\int_{\C^{*n}}(\phi_{1}-\phi_{0})\left((\omega^{\phi_{0}})^{n}-(\omega^{\phi_{1}})^{n}\right)\right)^{1/2^{(n-1)}}
\]
We will deduce this inequality from the following inequality of Blocki
\cite{bl} (see also\cite[Lemma A.1]{BBJ} for a more general inequality).
Let $(M,\omega)$ be a compact Kähler manifold and $u_{0}$ and $u_{1}$
$\omega-$psh functions on $M$ in $L^{\infty}(X).$ Then there exists
a constant $C_{M}$ such that 
\begin{equation}
i\int_{M}\partial(u_{0}-u_{1})\wedge\bar{\partial}(u_{0}-u_{1})\wedge\omega^{n-1}\leq C_{M}\left(\int_{M}(u_{0}-u_{1})(\omega_{u_{1}}^{n}-\omega_{u_{0}}^{n})\right)^{1/2^{(n-1)}}\label{eq:blocki}
\end{equation}
 The constant $C_{M}$ only depends (in a continuous fashion) on upper
bounds on the $L^{\infty}-$norms of $u_{0}$ and $u_{1}$ and the
volume of $\omega.$ More generally, exactly the same proof as in
\cite{bl} shows that the inequality holds more generally when $\omega$
is a semi-positive form with an $L^{\infty}-$potential and positive
volume. In the present setting we set 
\[
u_{0}:=\phi_{0}(z)-\phi^{Y^{j}}(z),\,\,\,u_{1}=\phi_{1}(z)-\phi^{Y^{j}}(z),
\]
 originally defined on $\C^{*n}.$ Just as in the proof of Theorem
\ref{thm:thm H1 intro} the functions $\phi_{0},\phi_{1}$ and $\phi^{Y^{j}}$
may be identified with $L^{\infty}-$metrics on the ample line bundle
$L\rightarrow M_{j}$ over the toric variety $M_{j}$ determined by
$Y_{j}.$ Moreover, the corresponding currents extend to positive
currents on $M_{j}.$ Passing to a smooth resolution we may as well
assume that $M_{j}$ is smooth (and $L$ semi-ample). Applying the
inequality \ref{eq:blocki} on $M_{j}$ all that remains is to verify
that the corresponding constants $C_{M_{j}}$ are under control. Applying
Lemma \ref{lem:l infty bound in terms of Sob} and using that, by
assumption, $Y_{j}$ is contained in $Y,$ gives, for a fixed $q>n,$
\[
\left\Vert \phi_{0}-\phi^{Y^{j}}\right\Vert _{L^{\infty}}\leq\frac{d(Y)}{V(Y_{j})}\int_{\R^{n}}|x|MA(\phi_{0})+C_{n,q}\frac{d(Y)^{\left(1+n(1-1/q)\right)}}{V(Y_{j})}\int_{\R^{n}}|x|^{q}MA(\phi_{0})+\left\Vert \phi^{Y_{j}}-\phi_{Y_{j}}\right\Vert _{L^{\infty}}
\]
 in terms of the $L^{\infty}-$norms on $\R^{n}.$ Hence, 
\[
\limsup_{j\rightarrow\infty}\left\Vert u_{0}\right\Vert _{L^{\infty}(M_{j})}\leq d(Y)\int_{\R^{n}}|x|\frac{MA(\phi_{0})}{V(Y)}+C_{n,q}d(Y)^{\left(1+n(1-1/q)\right)}\int|x|^{q}\frac{MA(\phi_{0})}{V(Y)}+\left\Vert \phi^{Y}-\phi_{Y}\right\Vert _{L^{\infty}},
\]
Since the probability measure $\mu_{0}:=MA(\phi_{0})/V(Y)$ is supported
in $X$ it follows that
\[
\limsup_{j\rightarrow\infty}\left\Vert u_{0}\right\Vert _{L^{\infty}(M_{j})}\leq d(Y)R(X)+C_{n,q}d(Y)^{\left(1+n(1-1/q)\right)}R(X)^{q}+\left\Vert \phi^{Y}-\phi_{Y}\right\Vert _{L^{\infty}}.
\]
 The same estimate also holds (for the same reason) for $u_{1}.$
Finally, since the left hand side in the inequality \ref{eq:ineq in pf prop key reg}
is invariant under $\phi^{Y}\rightarrow\phi^{Y}+c$ for $c\in\R$
it follows that the constant $C^{Y}$ appearing in \ref{eq:ineq in pf prop key reg}
only depends on upper bounds on $\phi_{Y}$ through the quantity $A_{3}$
defined by formula \ref{eq:def of A three}. Moreover, as explained
above $C^{Y}$ also depends on an upper bound on the volume of $\omega_{Y_{j}}$
on $M_{j}.$ But (by formula \ref{eq:snitt blir volym}) this equals
$n!$ times the volume of $Y_{j},$ which is bounded from above by
a constant times $d(Y)^{n}.$ All in all this proves the inequality
\ref{eq:ineq in pf prop key reg} with the desired control on the
constant $C_{Y}.$ 
\end{proof}

\subsubsection{The dependence on $X$ and $Y$}

Finally, we show that the constant $C$ appearing in Proposition \ref{prop:key non reg}
can be made to only depend on upper bounds on $d(X),$ $d(Y)$ and
$1/V(Y).$ To see this first note that both sides of the inequality
in Proposition \ref{prop:key non reg} are invariant under $\phi_{i}(x)\rightarrow\phi_{i}(x)-x\cdot b$
for any $b\in\R^{n},$ if $Y$ is replaced by $Y-\{b\}.$ In particular,
taking $b$ to be the barycenter of $Y$ we may as well assume that
$0$ is the barycenter of $Y.$ Similarly, since both sides of inequality
in Proposition \ref{prop:key non reg} are invariant under $\phi_{i}(x)\rightarrow\phi_{i}(x+a)$
for any $a\in\R^{n}$ we may as well assume that $0\in X.$ It will
then be enough to show that the auxiliary convex function $\phi^{Y}$
used in the proof above can be chosen so that the corresponding constants
$A_{1},A_{2}$ and $A_{3}$ only depend on upper bounds the outer
radius $R(X)$ and $R(Y)$ of $X$ and $Y$ (defined as in formula
\ref{eq:outer radius}) and $1/V(Y).$ Indeed, since $0\in Y$ and
$0\in X$ we have that $R(Y)\leq\ensuremath{d(Y)}$ and $R(X)\leq d(X).$

We will take $\phi^{Y}$ to be the unique smooth and strictly convex
function $\phi$ in $\mathcal{C}_{Y}(\R^{n})_{+}$ solving 
\begin{equation}
MA(\phi)=e^{-\phi}dx,\,\,\,(\nabla\phi)(0)=0\label{eq:keq equation}
\end{equation}
(the first equation geometrically means that $\phi(z)$ is a Kähler
potential for a Kähler-Einstein metric on $\C^{*n}$ and the second
one ensures that $\phi$ is uniquely determined). Since we have assumed
that $0$ is the barycenter of $Y$ it follows from \cite{be-be}
that such a solution indeed exists (and is uniquely determined). Moreover,
by \cite{kl} the solution satisfies 
\[
\nabla^{2}\phi\leq2R(Y)^{2}I.
\]
 Thus all that remains is to verify the bounds on $A_{2}$ and $A_{3}.$
To this end first note that it follows directly from the defining
equation for $\phi$ that it is enough to establish an upper bound
on $A_{3}.$ This will be accomplished by the following
\begin{prop}
The unique solution $\phi\in\mathcal{C}_{Y}(\R^{n})_{+}$ of \ref{eq:keq equation}
satisfies
\[
\left\Vert \phi-\phi_{Y}\right\Vert _{L^{\infty}(\R^{n})}\leq A,
\]
 where the constant $A$ only depends on upper bounds on $R(Y),$
$1/V(Y)$ and $1/R_{0}(Y),$ where $R_{0}(Y):=d(0,\partial Y)$ is
the distance between $0$ and $\partial Y.$ 
\end{prop}

\begin{proof}
This proposition is implicitly contained in \cite{do}, but since
it is not stated explicitly there we briefly explain how to extract
the bound in the proposition from the estimates in \cite[Section 3.3]{do}.
To this end denote by $v(y)$ the convex function on $Y$ defined
by the Legendre transform of $\phi$ (which is denoted by $u$ in
\cite{do} and $Y$ is denoted by $P$ there). The equation for $\phi$
translates into 
\[
\log\det(\nabla v)=x\cdot\nabla v-v
\]
(the lhs is denoted by $L$ in \cite{do} and the rhs by $h,$ but
the constant $C$ appearing in \cite[formula 22]{do} vanishes in
the present setting). The assumption that $\nabla\phi(0)=0$ equivalently
means that $\nabla v(0)=0$ which, in turn, means that the infimum
of $v$ on $Y$ is attained at $y=0.$ Set 
\[
m:=-v(0)=-\inf_{Y}v
\]
Step $2$ in \cite[Section 3.3]{do} yields a lower bound on $m$
in terms of an upper bound on $R(Y)$ and a lower bound on the volume
of $Y.$ Step 4 gives an upper bound on $m$ in terms of upper bounds
on $1/R_{0}(Y)$ and $1/V(Y).$ This implies that 
\[
\phi(x)\geq\kappa^{-1}|x|+m-\kappa^{-1},
\]
 where $\kappa$ is bounded from above by a constant only depending
on upper bounds on $1/R_{0}(Y)$ and $1/V(Y)$ (see \cite[formula 24]{do}).
But then it follows that, for any $q>0,$ the moment $\int_{\R^{n}}|x|^{q}MA(\phi)/V$
is bounded from above by a constant only depending on upper bounds
on $R(Y),$ $1/R_{0}(Y)$ and $1/V(Y).$ Hence, by Lemma \ref{lem:l infty bound in terms of Sob}
the inequality in the proposition holds if $\sup_{\R^{n}}(\phi-\phi_{Y})$
is added to the right hand side of the inequality. But, since $\sup_{\R^{n}}(\phi-\phi_{Y})=m,$
which as explained above is bounded from above by upper bounds on
$1/R_{0}(Y)$ and $1/V(Y),$ this proves the proposition.
\end{proof}
Finally, the desired bound on $A_{3}$ follows from the following
elementary lemma showing that $1/R_{0}(Y)\apprle R(Y)^{n-1}/V(Y):$
\begin{lem}
Let $Y$ be a convex body and denote by $b_{Y}$ the barycenter of
$Y.$ Assume that $Y$ is contained in a ball $B_{R}$ of radius $R.$
Then
\[
d(b_{Y},\partial Y)\geq\frac{1}{2c_{n}}\frac{V(Y)}{R^{n-1}},
\]
 where $c_{n}$ denotes the volume of the unit-sphere $\partial B_{1}$
in $\R^{n}.$ 
\end{lem}

\begin{proof}
Denote by $d(y)$ the function on $Y$ defined as the distance of
$y\in Y$ to $\partial Y$ and set $Y_{\epsilon}:=\{d\geq\epsilon\}$
for a given $\epsilon>0.$ Decomposing $b_{Y}$ as a convex combination
of $b_{Y_{\epsilon}}$ and $b_{(Y-Y_{\epsilon})},$ 
\[
b_{Y}=\frac{V(Y_{\epsilon})}{V(Y)}b_{Y_{\epsilon}}+\frac{V(Y-Y_{\epsilon})}{V(Y)}b_{(Y-Y_{\epsilon})}
\]
and using that the function $d(y)$ is concave (since it can be expressed
as the infimum over the affine functions defined by the distances
to the supporting hyperplanes of $Y$) and non-negative gives
\[
d(b_{Y})\geq\frac{V(Y_{\epsilon})}{V(Y)}\epsilon.
\]
By the monotonicity of the surface area of convex bodies (i.e the
classical fact that $V(\partial P)\leq V(\partial Q)$ if the convex
body $P$ is included in the convex body $Q$ \cite{st}) we have,
assuming that $Y-Y_{\epsilon}$ is non-empty, 
\[
V(Y_{\epsilon})=V(Y)-V(Y-Y_{\epsilon})\geq V(Y)-\epsilon V(\partial Y)\geq V(Y)-\epsilon V(\partial B_{R})
\]
Hence, $d(b_{Y})\geq\epsilon-\epsilon^{2}c_{n}R^{n-1}/V(Y)$ for any
such $\epsilon>0.$ Optimizing over $\epsilon,$ i.e. taking $\epsilon=\epsilon_{0}:=\frac{1}{2c_{n}}\frac{V(Y)}{R(Y)^{n-1}}$
thus concludes the proof when $Y-Y_{\epsilon}$ is non-empty. Finally,
if $Y-Y_{\epsilon_{0}}$ is empty then, since $b_{Y}\in Y,$ it must
be that $d(b_{Y})\geq\epsilon_{0},$which concludes the proof. 
\end{proof}

\subsection{The key analytic inequalities for general $g$}

We next turn to the case of a general density $g$ with a strict positive
lower bound $\delta$ on $Y$ (formula \ref{eq:def of delta in text}).
\begin{thm}
\label{thm:analytic ineqs text}Let $X$ and $Y$ be bounded open
domains $X$ in $\R^{n}$ and assume that $Y$ is convex and endowed
with a probability measure $\nu$ satisfying \ref{eq:def of nu}.
Let $\phi_{0}$ and $\phi_{1}$ be convex functions on $X,$ such
that the the closures of the sub-gradient images $(\partial\phi_{i})(X)$
are equal to $\overline{Y.}$ If there exists a positive constant
$C_{0}$ such that $\nabla^{2}\phi_{0}\geq C_{0}^{-1}I$ (in the sense
of distributions), then 
\[
\int_{X}|\nabla\phi_{0}-\nabla\phi_{1}|^{2}dx\leq n(n+1)C_{0}^{n-1}\delta^{-1}\int_{X}(\phi_{1}-\phi_{0})(MA_{g}(\phi_{0})-MA_{g}(\phi_{1})).
\]
 In general, there exists a constant $C$ independent of $\phi_{0}$
and $\phi_{1}$ such that 
\[
\int_{X}|\nabla\phi_{0}-\nabla\phi_{1}|^{2}dx\leq C\left(\delta^{-1}\int_{X}(\phi_{1}-\phi_{0})\left(MA_{g}(\phi_{0})-MA_{g}(\phi_{1})\right)\right)^{1/2^{n-1}}.
\]
 More precisely, the constant $C$ only depends on upper bounds on
the the diameters of $X$ and $Y,$ $d(X)$ and $d(Y)$ and a positive
lower bound on the volume $V(Y)$ of $Y$ 
\end{thm}

Fix $\phi_{0}\in\mathcal{C}_{Y}(\R^{n})_{+}$ (the class defined in
Section \ref{subsec:The-class-}) and set $\mu_{0}:=MA_{g}(\phi_{0}).$
Consider the following functional on $\mathcal{C}_{Y}(\R^{n})_{+}:$
\[
I_{g}(\phi):=\int_{\R^{n}}(\phi-\phi_{0})\left(-MA_{g}(\phi)+\mu_{0}\right)
\]
In order to prove Theorem \ref{thm:analytic ineqs text} it will,
by Propositions \ref{prop:key ineq}, \ref{prop:key non reg} be enough
to prove the following
\begin{prop}
\label{prop:red to}Setting $\delta:=\inf_{Y}g$ we have 
\[
I_{g}\geq\frac{\delta}{(n+1)}I
\]
\end{prop}

To this end consider the following auxiliary functional on $\mathcal{C}_{Y}(\R^{n})_{+}:$

\[
J_{g}(\phi):=\int_{Y}\phi^{*}gdy+\int(\phi-\phi_{0})\mu_{0}
\]
To simplify the notation we will write $I_{1}=I$ and $J_{1}=J.$
We start with a number of lemmas. 
\begin{lem}
\label{lem:The-Gateaux-differential}The Gateaux differential of $J_{g}$
at $\phi$ is represented by $-MA_{g}(\phi)+\mu_{0},$ i.e. fixing
$\phi_{1}$ and $\phi_{2}$ in $\mathcal{C}_{Y}(\R^{n})_{+}$ and
setting $\phi_{t}:=\phi_{1}+t(\phi_{2}-\phi_{1})$ we have
\[
\frac{dJ_{g}(\phi_{t})}{dt}_{|t=0}=\int_{\R^{n}}(\phi_{2}-\phi_{1})\left(-MA_{g}(\phi)+\mu_{0}\right)
\]
 Moreover, $t\mapsto J_{g}(\phi_{t})$ is convex and $J_{g}\geq0.$ 
\end{lem}

\begin{proof}
This is essentially well-known in the literature of optimal transport,
where $J_{g}$ appears as the Kantorovich functional (up to a change
of sign and and additive constant). A simple direct proof is given
in \cite{be-be}.
\end{proof}
\begin{rem}
\label{rem:aubin}The formula for the differential of $J_{g}$ implies
that when $g=1$ and $Y$ is a rational convex polytope, then $J_{g}$
coincides with Aubin's $J-$functional defined with respect to the
curvature form $\omega_{0}$ corresponding to $\phi_{0}$ (see \cite{aub}
for the smooth setting and \cite{bbgz} for the general singular setting). 
\end{rem}

Using the previous lemma we next establish the following 
\begin{lem}
\label{lem:I g bigger than J g}The following inequality holds: 
\[
I_{g}\geq J_{g}
\]
\end{lem}

\begin{proof}
Fix $\phi$ in $\mathcal{C}_{Y}(\R^{n})_{+}$ and set $\phi_{t}:=\phi_{0}+t(\phi-\phi_{0})$
and $I_{g}(t):=I_{g}(\phi_{t})$ and $J_{g}(t):=J_{g}(\phi_{t}).$
By the previous lemma we have, for any $t>0.$
\[
\frac{dJ_{g}(t)}{dt}=\int_{\R^{n}}(\phi-\phi_{0})\left(-MA_{g}(\phi_{t})+\mu_{0}\right)=t^{-1}I_{g}(t)
\]
 Hence, 
\[
\frac{dI_{g}(t)}{dt}=\frac{d}{dt}(t\frac{dJ_{g}(t)}{dt})=\frac{dJ_{g}(t)}{dt}+t\frac{d^{2}J_{g}(t)}{d^{2}t}\geq\frac{dJ_{g}(t)}{dt},
\]
 using that $J_{g}(t)$ is convex, by the previous lemma. Since $I_{g}(0)=J_{g}(0)(=0)$
it follows that $I_{g}(t)\geq J_{g}(t)$ for all $t\geq0$ and in
particular for $t=1,$ which concludes the proof.
\end{proof}
We next show that $J_{g}$ is controlled from below by $J:$
\begin{lem}
Setting $\delta:=\inf_{Y}g$ we have $J_{g}\geq\delta J\geq0$
\end{lem}

\begin{proof}
We continue with the notation in the proof of the previous lemma.
By a standard approximation argument we may as well assume that $\phi$
and $\phi_{0}$ are smooth and strictly convex. Moreover, as before
it will be enough to consider the case when $Y$ is a rational polytope.
First observe that it will be enough to prove the following claim:
\[
\text{Claim:\,}\frac{d^{2}J_{g}(t)}{d^{2}t}\geq\delta\frac{d^{2}J(t)}{d^{2}t}\geq0,
\]
Indeed, by Lemma \ref{lem:The-Gateaux-differential} the derivatives
of both $J_{g}(t)$ and $J(t)$ vanish at $t=0.$ As a consequence,
if the claim holds, then $\frac{dJ_{g}(t)}{dt}\geq\delta\frac{dJ(t)}{dt}$
for all $t$ and since $J_{g}(0)=J(0)=0$ it follows that $J_{g}(t)\geq\delta J(t)$
for all $t$ and in particular for $t=1,$ proving the lemma. In order
to prove the claim above it is enough to establish the following formula:
\begin{equation}
\frac{d^{2}J_{g}(\phi_{t})}{dt}=-\int_{\R^{n}}|\nabla\phi_{t}|_{g^{\phi_{t}}}^{2}MA_{g}(\phi_{t}).\label{eq:second der of J as dirichlet}
\end{equation}
To this end first consider a general curve $\phi_{t}(x)$ in $\mathcal{C}_{Y}(\R^{n})_{+}$
such that $\Phi(x,t):=\phi_{t}(x)$ is smooth on $\R^{n+1}.$ We complexify
$\R^{n}$ by $\C^{*n}$ (as in Section \ref{subsec:The-complex-torus})
and $\R$ by $\C$ (in the usual way). Denote by $(z,\tau)$ the corresponding
complex coordinates on $\C^{*n}\times\C.$ Denoting by $\pi$ the
natural projection from $\C^{*n}\times\C$ to $\C^{*n}$ we have the
following general formula
\[
\frac{\partial^{2}J_{g}(\phi_{\tau})}{\partial\tau\partial\bar{\tau}}id\tau\wedge d\bar{\tau}=-\pi_{*}(g(\nabla_{x}\pi^{*}\phi_{t})MA_{\C}(\Phi))
\]
on $\C,$ where $MA_{\C}(\Phi)$ denotes complex Monge-Ampère measure
of the smooth function $\Phi$ in $\C^{*n}\times\C.$ This formula
is a special case of the more general formula \cite[formula 2.11]{b-n}.
Next, we note that, if $\phi_{t}$ is affine, then $MA_{\C}(\Phi)\leq0.$
Indeed, a direct calculation gives 
\[
-MA_{\C}(\Phi)=\frac{i}{2\pi}(\pi^{*}\omega^{\phi_{t}})^{n}\wedge\partial_{z}(\phi-\phi_{0})\wedge\bar{\partial_{z}}(\phi-\phi_{0})\wedge id\tau\wedge d\bar{\tau}
\]
Using formula \ref{eq:Hessian norm as complex} this concludes the
proof of formula \ref{eq:sec bd cond intro}. 
\end{proof}

\subsubsection{Conclusion of proof of Prop \ref{prop:red to}}

Combing the previous lemmas with Lemma \ref{lem:I g bigger than J g}
gives $I_{g}\geq J_{g}\geq\delta J.$ Hence the proof is concluded
by noting that $J\geq I/(n+1).$ In the complex setting this is a
well-known inequality in the Kähler geometry literature which goes
back to Aubin \cite{aub}. For a simple proof see \cite[formula 2.7]{bbgz}.
The present real setting then follows from a complexification and
approximation argument, just as before. 

\subsection{Proof of Theorems \ref{thm:thm H1 intro}, \ref{thm:(General-case) intro} }

First note that the following simple lemma holds, where $\ensuremath{d(Y)}$
denotes the diameter of $Y:$
\begin{lem}
\label{lem:I g smaller than W one}Assume that $\phi_{0}$ and $\phi_{1}$
are convex functions on $\R^{n}$ whose gradient images are contained
in $\bar{Y}.$ Then 
\[
I_{g}(\phi_{0},\phi_{1})\leq d(Y)W_{1}\left(MA_{g}(\phi_{0}),MA_{g}(\phi_{1})\right).
\]
 
\end{lem}

\begin{proof}
This follows directly from the Kantorovich-Rubinstein formula
\begin{equation}
W_{1}(\mu_{0},\mu_{1}):=\sup_{u\in\text{Lip \ensuremath{_{1}(\R^{n})}}}\int u(\mu_{0}-\mu_{1})\label{eq:W_1 as sup}
\end{equation}
 where the sup runs over all Lipschitz continuous functions on $X$
with Lipschitz constant one (by taking $u=(\phi_{1}-\phi_{0})/\text{\ensuremath{d(Y)}).}$ 
\end{proof}
Applying the key inequalities in Theorem \ref{thm:analytic ineqs text}
to $\phi_{0}=\phi$ and $\phi_{1}=\phi_{h},$ and invoking the previous
lemma, all that remains is thus to verify that 
\begin{equation}
W_{1}(\mu,\mu_{h})\leq h.\label{eq:W one smaller than h}
\end{equation}
 To this end we rewrite the integral appearing in the right hand side
of formula \ref{eq:W_1 as sup} as 
\begin{equation}
\int u(\mu-\mu_{h})=\sum_{i=1}^{N}\int_{C_{i}}(u-u(x_{i})\mu\label{eq:integr u against difference mu}
\end{equation}
Since $u\in\text{Lip \ensuremath{_{1}(X)}}$ we trivially have
\[
x\in C_{i}\implies u(x)-u(x_{i})\leq d(C_{i})\leq h,
\]
 where the last inequality follows directly from the very definition
of $h.$ Hence, $W_{1}(\mu_{0},\mu_{1})$ is bounded from above by
$h$ times the total mass of $\mu,$ which proves the inequality \ref{eq:W one smaller than h}. 

Similarly, combining the analytic inequalities in Theorem \ref{thm:analytic ineqs text}
with Lemma \ref{lem:I g smaller than W one} and the inequality \ref{eq:W one smaller than h},
concludes the Proof of Theorems \ref{thm:thm H1 intro}, \ref{thm:(General-case) intro} 

\subsection{\label{subsec:Other-discretization-schemes}Alternative discretization
schemes when $\mu$ has a density}

Assume for simplicity that the domain $X$ has been normalized to
have unit volume and fix a sequence of point clouds $(x_{1},..x_{N})$
on $X.$ In the case when $\mu$ has a density $f$ it may, from a
computational point of view, by more convenient to consider discretizations
of $\tilde{\mu}_{h}$ of the form 
\begin{equation}
\tilde{\mu}_{h}:=\frac{\sum_{i=1}^{N}f(x_{i})w_{i}\delta_{x_{i}}}{\sum_{i=1}^{N}f(x_{i})w_{i}}:=\frac{f\lambda_{N}}{\int f\lambda_{N}}\label{eq:mu h tilde}
\end{equation}
 for appropriate weights $w_{i}$ (independent of $\mu).$ For example,
a convenient choice is to take $w_{i}$ to be equal to the volume
of the intersection with $X$ of the Voronoi cell corresponding to
$x_{i}.$ We claim that the result in Theorem \ref{thm:thm H1 intro}
then still holds if $f$ is Lipschitz continuous. Indeed, by the proof
of Theorem \ref{thm:thm H1 intro}, it is enough to show that 
\[
W_{1}(\mu,\tilde{\mu}_{h})\leq Ch,
\]
which, in turn, will follow from 
\[
W_{1}(dx,\lambda_{N})\leq h.
\]
But the latter inequality follows from the identity \ref{eq:integr u against difference mu}
applied to $\mu=dx.$ Similarly, if $f$ is merely assumed to be in
the Hölder space $C^{\alpha}(X)$ for $\alpha<1$ then the same argument
reveals the an analog of Theorem \ref{thm:thm H1 intro} holds, with
the rate $Ch^{1/2}$ replaced by $Ch^{\alpha/2}.$

Finally, in the case when the point cloud $\{x_{i}\}$ is equal to
the intersection of $X$ with the grid $(\Z N^{-1/n})^{n}$ and the
distance between the point cloud and the boundary of $X$ is of the
order $O(N^{-1/n}),$ then a slight variant of the previous argument
reveals that the same results as above hold, with $h:=N^{-1/n},$
when all weights $w_{i}$ in formula \ref{eq:mu h tilde} are taken
to be equal to $1/N.$

\section{\label{subsec:Relations-to-quantative}Quantitative stability of
optimal transport maps}

Denote by $\mathcal{P}_{ac}(\R^{n})$ the space of all probability
measures on $\R^{n},$ which are absolutely continuous with respect
to Lebesgue measure. Given $\mu$ and $\nu$ in $\mathcal{P}_{ac}(\R^{n})$
with compact support $\overline{X}$ and $\overline{Y},$ respectively,
there exist optimal transport maps $T_{\mu}^{\nu}$ and $T_{\nu}^{\mu}$
in $L^{\infty}(X,Y)$ and $L^{\infty}(Y,X)$ transporting $\mu$ to
$\nu$ and $\nu$ to $\mu,$ respectively \cite{br,v1}. The maps
are uniquely determined and inverses of each others almost every where
with respect to $\mu$ and $\nu,$ respectively. 

\subsection{Proof of Theorems \ref{thm:(Quantitative-stability--first intro},
\ref{thm:(Quantitative-stability--second intro}}

By Brenier's theorem \cite{br,v1} we can write $T_{\mu_{i}}^{\nu}=\nabla\phi_{i}$
for a convex function $\phi_{i}$ on $X$ such that $\overline{\partial\phi_{i}(X)}=\overline{Y.}$
The first inequality in Theorem \ref{thm:(Quantitative-stability--first intro}
then follows directly from the first inequality in Theorem \ref{thm:analytic ineqs text},
combined with the inequality in Lemma \ref{lem:I g smaller than W one}.
To prove the second inequality in \ref{thm:(Quantitative-stability--first intro}
first note that combining first inequality in Theorem \ref{thm:analytic ineqs text}
with the Cauchy-Schwartz inequalities yields
\begin{equation}
\left\Vert \nabla\phi_{0}-\nabla\phi_{1}\right\Vert _{L^{2}(X)}^{2}\leq n(n+1)C_{0}^{n-1}\delta^{-1}\left\Vert \phi_{0}-\phi_{1}\right\Vert _{L^{2}(X)}\left\Vert d\mu_{0}/dx-d\mu_{1}/dx\right\Vert _{L^{2}(X)}\label{eq:proof of quant}
\end{equation}
Applying, the Poincaré inequality to $\phi_{0}-\phi_{1}$ (whose integral
with respect to $(X,dx)$ may be assumed to vanish) gives $\left\Vert \phi_{0}-\phi_{1}\right\Vert _{L^{2}(X)}\leq C_{P}\left\Vert \nabla\phi_{0}-\nabla\phi_{1}\right\Vert _{L^{2}(X)}.$
Hence, dividing both sides of the inequality \ref{eq:proof of quant}
with $\left\Vert \nabla\phi_{0}-\nabla\phi_{1}\right\Vert _{L^{2}(X)}$
concludes the proof of Theorem \ref{thm:(Quantitative-stability--first intro}.
Finally, Theorem \ref{thm:(Quantitative-stability--second intro}
is proved in the same way, but now instead using the second inequality
in Theorem \ref{thm:analytic ineqs text}.

\subsection{\label{subsec:Sharpness-of-the}Sharpness of the exponent $1/2$
in the uniformly convex case}

We next show that the exponent $1/2$ in the first inequality in Theorem
\ref{thm:(Quantitative-stability--first intro} is optimal. We take
$X$ and $Y$ to both be equal to $(]0,1[)^{n}\subset\R^{n}$ and
$\nu=dy.$ 
\begin{lem}
There exists a curve $\phi_{\epsilon}$ of convex functions such that
$\overline{\partial\phi_{\epsilon}(X)}=\overline{Y}$ and $\phi_{0}$
is uniformly convex satisfying
\[
\left\Vert \nabla\phi_{\epsilon}-\nabla\phi_{0}\right\Vert _{L^{2}(X)}=C\left(W_{1}\left(MA(\phi_{\epsilon}),MA(\phi_{0})\right)\right)^{1/2}
\]
for a constant $C>0$ independent of $\epsilon.$ Moreover, $W_{1}\left(MA(\phi_{\epsilon}),MA(\phi_{0})\right)$
is non-zero and tends to zero as $\epsilon\rightarrow0.$ 
\end{lem}

\begin{proof}
First consider the case when $n=1.$ Fix a smooth convex function
$\phi_{1}$ on $[0,1],$ not equal to $|x|+c$ for any constant $c,$
such that $\overline{\partial\phi_{1}(]0,1[)}=[0,1]$ and define
\[
\phi_{\epsilon}(x):=\frac{1}{2}\epsilon\phi_{1}(\frac{x}{\epsilon})+\frac{1}{2}|x|^{2},\,\,\,\phi_{0}(x):=\frac{1}{2}|x|+\frac{1}{2}|x|^{2},
\]
where the second term ensures uniform convexity. Denote by $S_{\epsilon}$
the scaling map on $\R$ defined by $x\mapsto\epsilon x.$ Then, for
$\epsilon>0,$ 
\[
2\left(MA(\phi_{\epsilon})-MA(\phi_{0})\right)=(S_{\epsilon})_{*}(MA(\phi_{1})-\delta_{0}.
\]
Hence, using the dual representation \ref{eq:W_1 as sup} for $W_{1},$
\[
W_{1}\left(MA(\phi_{\epsilon}),MA(\phi_{0})\right)=\frac{1}{2}W_{1}\left((S_{\epsilon})_{*}(MA(\phi_{1}),\delta_{0}\right)=\frac{1}{2}\int|x|(S_{\epsilon})_{*}(MA(\phi_{1}))=\frac{\epsilon}{2}\int|x|MA(\phi_{1}).
\]
 On the other hand, denoting by $H(x)$ the sign of $x,$ we have
$\left\Vert \nabla\phi_{\epsilon}-\nabla\phi_{0}\right\Vert _{L^{2}[0,1]}^{2}=$
\[
=\frac{1}{2^{2}}\int_{0}^{1}\left|\nabla_{x}\epsilon\phi_{1}(\frac{x}{\epsilon})-H(x)\right|^{2}dx=\frac{1}{2^{2}}\int_{0}^{1}\left|\phi_{1}'(\frac{x}{\epsilon})-H(x)\right|^{2}dx=\epsilon\frac{1}{2^{2}}\int_{0}^{1}\left|\phi'(x)-H(x)\right|^{2}dx.
\]
 This proves the equality in the lemma when $n=1.$ Finally, for $n>1$
we simply set 
\[
\psi_{\epsilon}(x):=\phi_{\epsilon}(x_{1})+\sum_{i=2}^{n}\frac{1}{2}|x_{i}|^{2},\,\,\,\psi_{0}(x):=\phi_{0}(x)+\sum_{i=2}^{n}\frac{1}{2}|x_{i}|^{2}.
\]
 Then $\left\Vert \nabla\psi_{\epsilon}-\nabla\psi_{0}\right\Vert _{L^{2}([0,1]^{n})}=\left\Vert \nabla\phi_{\epsilon}-\nabla\phi_{0}\right\Vert _{L^{2}([0,1])}.$
Moreover, since $MA(\psi_{\epsilon})=\partial_{x_{1}}^{2}\phi_{\epsilon}(x_{1})dx$
it follows from the representation \ref{eq:W_1 as sup} that $W_{1}\left(MA(\psi_{\epsilon}),MA(\psi_{0})\right)=W_{1}\left(MA(\phi_{\epsilon}),MA(\phi_{0})\right)_{\R}.$
Hence, the curve $\psi_{\epsilon}$ satisfies the equality in the
lemma on $X=(]0,1[)^{n}.$ 

Now, assume that the first inequality in Theorem \ref{thm:(Quantitative-stability--first intro}
holds with the exponent $1/2$ replaced by $\alpha$ for some given
$\alpha>0$ and with a constant $c$ only depending on a strict power
lower bound on the uniform convexity of the convex potential of $T_{0}.$
This implies that there exists a constant $C_{0},$ only depending
on $\phi_{0},$ such that 
\[
\left\Vert \nabla\phi_{1}-\nabla\phi_{0}\right\Vert _{L^{2}(X)}\leq C\left(W_{1}\left(MA(\phi_{1}),MA(\phi_{0})\right)\right)^{\alpha}
\]
for any convex functions $\phi_{i}$ such that $\overline{\partial\phi_{\epsilon}(X)}=\overline{Y}$
with $\phi_{0}$ uniformly convex. Indeed, this follows from applying
the inequality in question to regularizations of $\nabla\phi_{0}$
and $\nabla\phi_{1}.$ But then the previous lemma forces $\alpha\leq1/2,$
as desired. 
\end{proof}

\subsection{\label{subsec:Variations-with-respect}Variations with respect to
the source vs the target measure}

The two different types of quantitative stability results discussed
in Sections \ref{subsec:Comparison-with-previous}, \ref{subsec:Relations-to-duality}
can be related thanks to the following lemma, essentially contained
in the proof of \cite[Cor 2.6]{m-d-c}:
\begin{lem}
Assume that there exist $\alpha_{1}>0$ and $C_{1}$ such that for
any $\mu_{0},\mu_{1}\in\mathcal{P}_{ac}(X)$ 
\[
\left\Vert T_{\mu_{0}}^{\nu}-T_{\mu_{1}}^{\nu}\right\Vert _{L^{2}(X,dx)}\leq C_{1}W_{1}(\mu_{0},\mu_{1})^{\alpha_{1}}.
\]
 Then there exists $A_{1}>0$ such that for any any $\mu_{0},\mu_{1}\in\mathcal{P}_{ac}(X),$
\[
\left\Vert T_{\nu}^{\mu_{0}}-T_{\nu}^{\mu_{1}}\right\Vert _{L^{2}(Y,dy)}\leq A_{1}W_{1}(\mu_{0},\mu_{1})^{\alpha_{1}/(n+2)}
\]
Conversely, assume that there exist $\alpha_{2}>0$ and $A_{2}>0$
such that for any $\mu_{0},\mu_{1}\in\mathcal{P}_{ac}(X)$
\[
\left\Vert T_{\nu}^{\mu_{0}}-T_{\nu}^{\mu_{1}}\right\Vert _{L^{2}(X,dx)}\leq A_{2}W_{1}(\mu_{0},\mu_{1})^{\alpha_{2}}.
\]
 Then there exists a constant $C_{2}>0$ such that for any $\mu_{0},\mu_{1}\in\mathcal{P}_{ac}(X)$
\[
\left\Vert T_{\mu_{0}}^{\nu}-T_{\mu_{1}}^{\nu}\right\Vert _{L^{2}(X,dx)}\leq C_{2}W_{1}(\mu_{0},\mu_{1})^{\alpha_{2}/(n+2)}
\]
\end{lem}

\begin{proof}
Recall that, using the notation in the proof of Theorems \ref{thm:(Quantitative-stability--first intro},
\ref{thm:(Quantitative-stability--second intro} above, we can express
$T_{\mu_{0}}^{\nu}=\nabla\phi_{i}$ and $T_{\nu}^{\mu_{0}}=\nabla\phi_{i}^{*},$
where $\phi^{*}$ denotes the Legendre transform of $\phi.$ The proof
of the lemma is now obtained by combining the fact that the Legendre
transform preserves the $L^{\infty}-$norm together with the following
two well-known inequalities on a bounded Lipschitz domain $\Omega\subset\R^{n}$
for Lipschitz continuous functions $u_{0}$ and $u_{1}$ with Lip
constant one, assumed convex in the second inequality: 
\[
\left\Vert u_{0}-u_{1}\right\Vert _{L^{\infty}(\Omega)}\leq A_{1}\left\Vert u_{0}-u_{1}\right\Vert _{L^{2}(\Omega)}^{2/(n+2)},\,\,\,\left\Vert \nabla u_{0}-\nabla u_{1}\right\Vert _{L^{2}(\Omega)}\leq A_{2}\left\Vert u_{0}-u_{1}\right\Vert _{L^{\infty}(\Omega)}^{1/2}
\]
together with the Poincaré inequality on $\Omega.$ We recall that
the first inequality for $u_{0}-u_{1}$ above is a special case of
the Gagliardo\textendash Nirenberg inequalities (see \cite[Lemma 5.2]{lo}
for a direct simple proof) and the second one is shown in \cite[Thm 22]{c-c-s-l}).
\end{proof}

\section{\label{sec:Formulation-in-terms}Formulation in terms of computational
geometry}

In this section we show how to use Theorem \ref{thm:thm H1 intro}
to obtain a quantitative approximation of the solution $\phi$ to
the second boundary problem for the Monge-Ampère operator, using computational
geometry, as developed in \cite{au,a-h-a,me,gu et al,l-s,k-m-t}.
A comparison with the standard notation from computational geometry
and semi-discrete optimal transport is provided in Section \ref{subsec:Relations-to-duality}
and Section \ref{subsec:Comparison-with-semi-discrete}, respectively.

Assume given a discrete probability measure on $\R^{n}$ and a convex
bounded domain $Y$ of unit volume. In order to facilitate the comparison
with the setup considered in the previous sections (in particular
Section \ref{subsec:Discretization-using-semi-discre}) the discrete
probability measure in question will be denoted by $\mu_{h}:$ 

\[
\mu_{h}=\sum_{i=1}^{N}f_{i}\delta_{x_{i}},\,\,\,\,x_{i}\in\R^{n},f_{i}:=\mu_{h}(\{x_{i}\}).
\]
 The index $h$ is thus used a reminder that the measure $\mu_{h}$
is discrete. Let $\phi_{h}$ be the (finite) convex function on $\R^{n}$
(uniquely determined up to an additive constant) by 
\[
MA_{g}(\phi_{h})=\mu_{h},\,\,\,\,\,\overline{(\partial\phi_{h})(\R^{n})}=\overline{Y}
\]
Set 
\[
\boldsymbol{\phi}_{h}:=(\phi_{h}(x_{1}),...,\phi_{h}(x_{N}))\in\R^{N},
\]
 which we identify with a function on $\R^{n}$ which is equal to
$\phi_{h}$ on $\{x_{1},...,x_{N}\}$ and $\infty$ elsewhere. We
will denote by $X_{h}$ the bounded convex polytope defined as the
convex hull of $\{x_{1},...,x_{N}\}:$ 
\[
X_{h}:=\text{conv}\{x_{1},..,x_{N}\}
\]
We will refer to the following equation for a vector $\boldsymbol{\phi\in}\R^{N}$
as the \emph{discrete Monge-Ampère equation }(associated to the discrete
measure $\mu_{h}$ on $\R^{n}$ and the measure $\nu$ on $Y):$ 
\begin{equation}
\text{Vol \ensuremath{_{\nu}(F_{i}^{\boldsymbol{\phi}^{*}}\cap Y)=f_{i}}}\label{eq:f i eqial integral over F i}
\end{equation}
where $(F_{i}^{\boldsymbol{\phi}^{*}})_{i=1}^{N}$ are the cells in
the partition of $\R^{n}$ induced by the piecewise affine function
$\boldsymbol{\phi}^{*}$(see point 6 and 7 in Section \ref{subsec:Recap-of-the}).

By the following proposition the previous equation is solved by the
vector $\boldsymbol{\phi}_{h}$ and moreover, the graph of $\phi_{h}$
over $X_{h}$ is simply the convex hull of the discrete graph over
$\{x_{1},..,x_{N}\}$ of $\boldsymbol{\phi}_{h}.$ 
\begin{prop}
\label{prop:structure of sol over X h}Denote by $\psi_{h}:=\phi_{h}^{*}$
the Legendre transform of $\phi.$ Then 
\begin{itemize}
\item The vector $\boldsymbol{\phi}_{h}$ solves the discrete Monge-Ampère
equation \ref{eq:f i eqial integral over F i} and on $\overline{Y}$
the convex function $\psi_{h}$ is piecewise affine and given by
\begin{equation}
\psi_{h}(y)=\boldsymbol{\phi}_{h}^{*}(y):=\left(\max_{\{x_{i}\}}x_{i}\cdot y-\phi(x_{i})\right)\,\,\,\,\,y\in\overline{Y}\label{eq:psi h as leg of discr}
\end{equation}
\item On the convex hull $X_{h}$ of $\{x_{1},...,x_{N}\}$ the convex function
$\phi_{h}$ is piecewise affine and given by 
\begin{equation}
\phi_{h}(x)=\left(\max_{\{y_{i}\}}x\cdot y_{i}-\psi_{h}(y_{i})\right)\,\,\,\,\,x\in X_{h},\label{eq:phi h as leg of discr}
\end{equation}
 where $\{y_{i}\}_{i=1}^{M}$ runs over the vertices in the polyhedral
decomposition of $Y$ determined by the convex piecewise affine function
$\psi_{h}$ on $Y.$ Moreover, the graph of $\phi$ over $X_{h}$
is equal to the convex hull of the graph of the discrete function
$\phi(x_{i})$ over $X_{h}.$ 
\item The following duality relations hold: 
\[
\overline{(\partial\phi_{h})(X_{h})}=\overline{Y}
\]
 and the multi-valued maps $\partial\phi_{h}$ and $\partial\psi_{h},$
which are inverses to each-other, yield a bijection between the facets
and the vertices of the induced tessellations of $X_{h}$ and $\overline{Y}.$
\end{itemize}
\end{prop}

\begin{proof}
This is without doubt essentially well-known (and it is, as discussed
below, closely related to results in computational geometry). But
for completeness a simple analytic proof is provided using the basic
properties of the Legendre transform recalled in Section \ref{subsec:Recap-of-the},
point $1-8.$ First observe that 
\[
\overline{(\partial\phi_{h})(\R^{n})}=\overline{(\partial\phi_{h})(\{x_{i}\})}=\overline{Y},
\]
 as follows directly from the defining property of $\phi_{h}$ combined
with point 4. Hence, for almost any $y\in\overline{Y}$ there exists
$x_{i}$ such that $y\in\partial\phi(x_{i}).$ But then formula \ref{eq:psi h as leg of discr}
on $Y$ follows from point $5.$ Now formula \ref{eq:f i eqial integral over F i}
follows from the very definition $MA(\phi_{h})[F_{i}^{\psi_{h}}].$ 
\end{proof}
To prove formula \ref{eq:phi h as leg of discr} on $X_{h}$ it will,
by the previous argument (with $\phi_{h}$ replaced by $\psi_{h})$
be enough to show that 
\begin{equation}
\overline{(\partial\psi_{h})(Y)}=\overline{(\partial\psi_{h})(\{y_{i}\})}=X_{h}.\label{eq:grad image psh h in pf}
\end{equation}
 Denote by $\tilde{\psi}_{h}$ the canonical piecewise affine extension
of $\psi_{h|Y}$ to $\R^{n}$ (defined by formula\ref{eq:phi h as leg of discr}).
Note that 
\[
\overline{(\partial\tilde{\psi}_{h})(\R^{n})}=X_{h}.
\]
 Indeed, by definition, $\tilde{\psi}_{h}=\boldsymbol{\phi}_{h}^{*}$
on all of $\R^{n}$ and hence this follows from point 3. All that
remains is thus to verify that $\overline{(\partial\tilde{\psi}_{h})(\R^{n})}=\overline{(\partial\tilde{\psi}_{h})(Y).}$
To this end it is, by point 4, enough to show that $MA(\tilde{\psi}_{h})$
is supported in $Y.$ By point 8 this amounts to showing that the
vertices $y_{i}$ defined by $\tilde{\psi}_{h}$ are all contained
in $Y.$ Assume, to get a contradiction, that this is not the case.
Then there exists a facet $F_{i}$ of $\tilde{\psi}_{h}$ which does
not intersect $Y$ (by the hyperplane separation theorem for convex
sets). But this contradicts formula \ref{eq:f i eqial integral over F i}.
This proves formula \ref{eq:grad image psh h in pf} and hence formula
\ref{eq:phi h as leg of discr}, as well. All in all, this means that
\[
\phi_{h}=(\chi_{K}+\phi_{h})^{**}\,\,\,\text{on \,\ensuremath{X_{h}},\,\,\ensuremath{K:=\{x_{i}\}_{i=1}^{N}}}.
\]
 But then it follows, from general principles, that on $\R^{n}$ 
\[
\phi_{h}(x)=\sup\left\{ \Phi(x):\,\,\text{\ensuremath{\Phi\,}convex\,on \ensuremath{\R^{n},\,}}\Phi\leq\phi_{h}\,\text{on \,\ensuremath{K}}\right\} 
\]
Now, denote by $\bar{\phi}$ the convex function on $X_{h}$ whose
graph is the convex hull of the discrete graph of $\boldsymbol{\phi}_{h}.$
Equivalently, this means that
\[
x\in X_{h}\implies\bar{\phi}(x)=\sup\left\{ \Phi(x):\,\,\text{\ensuremath{\Phi\,}convex\,on \ensuremath{X_{h},\,}}\Phi\leq\phi_{h}\,\text{on \,\ensuremath{K}}\right\} 
\]
Hence, $\chi_{K}\bar{\phi}\leq\phi_{h}$ on $\R^{n}$ and $\phi_{h}\leq\bar{\phi}$
on $X_{h}$ which implies that $\phi_{h}=\bar{\phi}$ on $X_{h},$
i.e. the graph of $\phi_{h}$ is the convex hull of the discrete graph
of $\boldsymbol{\phi}_{h},$ as desired. 
\begin{rem}
If $Y$ is a polytope it follows from classical results that $\phi_{h}$
is piecewise affine on all of $\R^{n}$ (see (see \cite[Section 17.2]{ba}).
However, this is not true in general, as illustrated by the case when
$Y$ is the unit-ball and $\mu_{h}=\delta_{0}$, so that $\phi_{h}(x)=|x|.$
\end{rem}

By the next lemma the solution $\boldsymbol{\phi}_{h}$ to the discrete
Monge-Ampère equation\ref{eq:f i eqial integral over F i} is, in
fact, the unique solution (mod $\R):$ 
\begin{lem}
\label{lem:unique disc ma eq}The equation \ref{eq:f i eqial integral over F i}
admits a unique solution in $\R^{n}/\R.$ 
\end{lem}

\begin{proof}
This is a consequence of the results in \cite{gu et al,k-m-t}: the
solutions are critical point of the Kantorovich functional on $\R^{N},$
which in \cite{k-m-t,gu et al} is shown to be strictly concave on
the subset $\mathcal{K}_{+}$ of $\R^{N}/\R$ defined as all $\boldsymbol{\phi}$
such that $\text{Vol}_{\nu}\ensuremath{(F_{i}^{\boldsymbol{\phi}^{*}}\cap Y)>0}$
for $i=1,..,N$ (using properties of graph Laplacians). But it may
be illuminating to point out that the uniqueness result is also a
consequence of the following general result (applied to $X=\{x_{1},...x_{N}\}$
and $\mu=\mu_{h}).$ Let $\nu$ be a probability measure of the form
$g_{Y}dy$ with $g>0$ and $\mu$ a probability measure on $\R^{n}$
with support compact $X.$ Assume that $\phi$ is a function on $X$
such that $(\nabla(\chi_{X}+\phi)^{*})_{*}\nu=\mu.$ Then $\phi$
is uniquely determined (mod $\R)$ a.e. with respect to $\mu$ (and
moreover, equal to the restriction to $X$ of a convex function on
$\R^{n}$ with subgradient image in $Y)$$.$ Indeed, setting $\psi:=(\chi_{Y}+\phi)^{*}$
the argument in the proof of Lemma \ref{lem:unique disc ma eq} gives
that the convex function $P^{Y}\phi:=(\chi_{Y}+\psi)^{*}(=(\chi_{Y}+\phi)^{*})$
on $\R^{n}$ is uniquely determined and satisfies $MA_{g}(P^{Y}\phi)=\mu.$
But, by general principles, we have that 
\begin{equation}
MA(P^{Y}\phi)=0\,\,\text{on\,}\{P^{Y}\phi<\phi\}\label{eq:vanishing of ma}
\end{equation}
and hence $P^{Y}\phi=\phi$ on $X,$ showing that $\phi$ is uniquely
determined (mod $\R)$ on $X.$ The vanishing \ref{eq:vanishing of ma}
can, for example, by shown by noting that $P^{Y}\phi$ is the upper
envelope of all functions $\varphi$ in $\mathcal{C}_{Y}(\R)$ dominated
by $\phi.$ Then \ref{eq:vanishing of ma} follows from a local argument
using the maximum principle for the Monge-Ampère operator (just as
in the more general case of the complex Monge-Ampère operator considered
in \cite[Prop 2.10]{ber-bo}).
\end{proof}
There has been extensive numerical work on constructing solutions
to the discrete Monge-Ampère equation \ref{eq:f i eqial integral over F i}
on $\R^{N},$ based on computational geometry \cite{au,a-h-a,me,l-s,k-m-t}.
In particular, it was shown in \cite{k-m-t} that the solution $\boldsymbol{\phi}_{h}\in\R^{N}$
is the limit point of a damped Newton iteration, which converges globally
at a linear rate (and locally at quadratic rate if $g$ is Lipschitz
continuous$)$. Performing one step in the iteration, $\boldsymbol{\phi}^{(m)}\mapsto\boldsymbol{\phi}^{(m+1)},$
amounts to computing the convex hull in $\R^{n+1}$ of the discrete
graph over $\{x_{1},...x_{N}\}$ defined by $\boldsymbol{\phi}^{(m)}$
(or equivalently, the corresponding power diagram in $\R^{n}$ and
its intersection with $Y;$ see Section \ref{subsec:Relations-to-duality}
below). 

Here we show that Theorems \ref{thm:thm H1 intro}, \ref{thm:(General-case) intro}
 yield a quantitative rate of convergence, in the limit when the spatial
resolution $h\rightarrow0,$ for the solutions $\boldsymbol{\phi}_{h}$
and the corresponding piecewise constant maps $T_{h}$ from $X_{h}$
to $\R^{n}$ defined as follows. Consider the tessellation of $X_{h}$
by the convex polytopes $F_{j}^{*}$ obtained by projecting to $\R^{n}$
the facets of the ``upper boundary'' of the convex hull in $\R^{n+1}$
of the discrete graph $\Gamma_{h}$ of the solution $\boldsymbol{\phi}_{h}\in\R^{N}.$
Then $T_{h}$ maps $F_{j}^{*}$ to the dual vertex $y_{j}\in\R^{n}.$
In geometric terms, $y_{j}$ is the projection to $\R^{n}$ of the
corresponding normal of the graph $\Gamma_{h}$ in $\R^{n+1}$ (normalized
so that its projection to $\R$ is $-1).$ 

We now come back to the setup introduced in Section \ref{subsec:Convergence-rates-for},
where $T_{\mu}^{\nu}$ denotes the optimal $L^{\infty}-$map transporting
$\mu$ to $\nu$ and $h$ denotes the spatial resolution of the discretization
$\mu_{h}$ of $\mu.$
\begin{thm}
\label{thm:computional geometr}Assume that $(X,Y,\mu,\nu)$ is regular
(or more generally, that the potential $\phi$ of $T_{\mu}^{\nu}$
is uniformly convex) and let $\boldsymbol{\phi}_{h}\in\R^{N}$ be
the normalized solution to the discrete Monge-Ampère equation \ref{eq:f i eqial integral over F i}
and denote by $T_{h}(=\nabla\phi_{h})$ the corresponding piecewise
constant maps from $X_{h}$ to $\R^{n}.$ Then there exists a constant
$C,$ depending on $(X,Y,\mu,\nu),$ such that
\[
\left\Vert T_{h}-T_{\mu}^{\nu}\right\Vert _{L^{2}(X_{h})}^{2}:=\sum_{j=1}^{M}\int_{F_{j}^{*}}|y_{j}-T_{\mu}^{\nu}(x)|^{2}dx\leq Ch
\]
where $T_{\mu}^{\nu}$ is the optimal diffeomorphism transporting
$\mu$ to $\nu.$ In the general case when $X$ and $Y$ are bounded
domains in $\R^{n}$ with $Y$ assumed convex and $\nu$ a probability
measure of the form \ref{eq:def of nu} the estimates above hold if
$h$ is replaced by $h^{1/2^{(n-1)}}$ and then the constant $C$
only depends on upper bounds on the the diameters of $X$ and $Y$
and on a positive lower bound on $\delta(:=\sup_{Y}(\nu/dy)).$ 
\end{thm}

\begin{proof}
Combing Theorems \ref{thm:thm H1 intro}, \ref{thm:(General-case) intro}
with Prop \ref{prop:structure of sol over X h} and Lemma \ref{lem:unique disc ma eq}
immediately yields the theorem. 
\end{proof}

\subsection{\label{subsec:Relations-to-duality}Comparison with duality in computational
geometry}

Proposition \ref{prop:structure of sol over X h} is closely related
to the well-known duality in computational geometry between \emph{weighted
Voronoi tessellations} (also knows as Laguerre tessellations and power
diagrams) of $\R^{n}$ and \emph{weighted Delaunay tessellations \cite{au}}
of convex polytopes (see also \cite{gu et al}). To briefly explain
this we first note that, by definition, the facets $F_{i}^{\boldsymbol{\phi}_{h}^{*}}$
in Proposition \ref{prop:structure of sol over X h} have the property
that $\boldsymbol{\phi}_{h}^{*}(y)$ is affine on $F_{i}^{\boldsymbol{\phi}_{h}^{*}}$
and equal to $x_{i}\cdot y-\phi_{h}(x_{i})$ there (which is the affine
function realizing the max over $x_{j}$ in formula \ref{eq:phi h as leg of discr}).
Completing the square this means that

\[
F_{i}^{\boldsymbol{\phi}_{h}^{*}}=\left\{ y\in\R^{n}:\,\,|x_{i}-y|^{2}+w_{i}\leq|x_{j}-y|^{2}+w_{j}\,\forall j\right\} \,\,\,\,w_{i}:=2\phi_{h}(x_{i})-|x_{i}|^{2},
\]
which is the definition of the cells in the weighted Voronoi tessellation
of $\R^{n}$ associated to the weighted points $(x_{i};w_{i})_{i=1}^{N})$
\cite{au}. The corresponding weighted Delaunay tessellation is defined
as the projection to $X_{h}$ of the polyhedral cell-complex in $\R^{n+1}$
defined by the convex hull of the graph of $\boldsymbol{\phi}_{h}.$
The duality in Proposition \ref{prop:structure of sol over X h} implies
that the corresponding sub-gradient map $\partial\phi_{h}$ maps the
facets $F_{i}^{*}$of the weighed Delaunay tessellation of $X_{h}$
to the vertices $y_{i}$ of the weighted Voronoi tessellation of $\R^{n}.$
More generally, $p-$dimensional faces correspond to $(n-p)-$dimensional
faces \cite{au}.
\begin{rem}
By the duality above, computing a convex hull of a discrete graph
over $N$ points in $\R^{n}$ is equivalent to computing the corresponding
weighted Voronoi tessellation (as emphasized in\emph{ \cite{au}}).
For example, when $n=2$ the time-complexity is $O(N\log N).$ Moreover,
if $g$ is constant and $Y$ is a polytope then the computation of
the volumes of the corresponding cells has time-complexity $O(N)$
(since the cells are convex polytopes). More generally, efficient
computation of the volume can be done if the density $g$ is piecewise
affine. However, the Newton iteration referred to above (as opposed
to a steepest descent iteration) also requires computing the volumes
of the intersections of the cells \cite{me,l-s}. This has worst-case
time-complexity $O(N^{2}),$ but the experimental findings in \cite{me,l-s}
indicate much better near linear time-complexity.
\end{rem}

\subsection{\label{subsec:Comparison-with-semi-discrete}Comparison with semi-discrete
optimal transport}

While the $L^{\infty}-$map $T:=\nabla\phi$ is the optimal map pushing
forward the measure $\mu\in\mathcal{P}_{as}(X)$ to $\nu\in\mathcal{P}_{ac}(Y),$
\[
T:\,X\rightarrow Y,\,\,\,T_{*}:\,\mu\mapsto\nu
\]
 the push-forward of the discrete measure $\mu_{h}$ under the piece-wise
constant map 
\[
T_{h}:\,X_{h}\rightarrow Y
\]
 is not even well-defined. However, the inverse $(T_{h})^{-1}$ of
$T_{h}$ is a well-defined $L^{\infty}-$map from $Y$ to $X_{h},$
pushing forward $\nu$ to the discrete measure $\mu_{h}$ on $X_{h}$.
In fact, $(T_{h})^{-1}$ is the optimal such map, as follows from
standard results in the theory of semi-discrete optimal transport
(indeed, $T_{h}^{-1}=\nabla\phi_{h}^{*},$ where $\phi_{h}^{*}$ is
the Legendre transform of $\phi_{h}).$ However, the notation used
here differs from the standard notation adopted in the literature
on semi-discrete optimal transport \cite{k-m-t}. Indeed, in the latter
case the measure $\nu$ is viewed as the source measure. Accordingly,
the space that is here called $X$ is called $Y$ in \cite{k-m-t}
and the measure $\mu$ that is treated as a source measure here is
viewed as the target measure in \cite{k-m-t}, where it is called
$\nu.$ Moreover, the function $-\phi_{h}(x)+|x|^{2}/2$ on $X,$
in the present notation, corresponds to the function $\psi$ in the
notation of \cite{k-m-t} (when the cost function is taken to be $|x-y|^{2}/2).$
The mismatch between the notation here and the one in \cite{k-m-t}
is a reflection of the fact that here the emphasis is put on the solutions
of the Monge-Ampère equations (and their discretizations), rather
than on semi-discrete optimal transport maps.

\section{\label{sec:The-periodic-setting}The periodic setting }

Assume given a $\Z^{n}-$periodic measures $\mu$ and $\nu$ on $\R^{n}$
normalized so that their total mass on a (or equivalently any) fundamental
region is equal to one. We will assume that $\nu$ has a positive
density $g$ with positive lower bound $\delta>0.$ We will identify
$\mu$ and $\nu$ with probability measures on the torus $M:=(\frac{\R}{\Z})^{n}.$
We endow $M$ with its standard flat Riemannian metric induced from
the Euclidean metric on $\R^{n}$ and denote by $dx$ the corresponding
volume form on $M.$ Setting

\[
u(x):=\phi(x)-|x|^{2}/2
\]
gives a bijection between quasi-convex functions $\phi$ on $\R^{n}$
(i.e. such that $\partial\phi$ is periodic) and functions $u\in C^{0}(M)$
which are quasi-convex in the sense that $\nabla^{2}u+I\geq0$ holds
in the weak sense of currents.

When $u\in C^{2}(M)$ and $\mu=fdx$ the Monge-Ampère \ref{eq:ma eq with mu intro}
for a quasi-convex function $\phi$ on $\R^{n}$ is equivalent to
the following equation for $u:$ 
\[
g(I+\nabla u(x))\det(I+\nabla^{2}u)dx=fdx
\]
 We will say that $(\mu,\nu)$ is regular if such a solution exists.
This is the case if both $f$ and $g$ are Hölder continuous and strictly
positive \cite{c-v}. In the case of general $(\mu,\nu)$ there always
exists a weak solution, in the sense of Alexandrov, which is unique
modulo an additive constant \cite{c-e}.

Let now $\mu_{h}$ be a family of discrete measures on $M,$ defined
as in formula \ref{eq:def of mu h} (with $h$ denoting the corresponding
mesh norm). Then the following analog of Theorem \ref{thm:thm H1 intro}
holds:
\begin{thm}
Let $M$ be the $n-$dimensional standard flat torus and $u$ and
$u_{h}$ quasi-convex functions on $M,$ defined as above. Then, in
the regular case, there exists a constant $C$ such that
\[
\left\Vert u_{h}-u\right\Vert _{H^{1}(M)}:=\left(\int_{M}|\nabla u_{h}-\nabla u|^{2}dx\right)^{1/2}\leq Ch^{1/2}
\]
The constant $C$ only depends on $f$ through a positive lower bound
on $(\nabla^{2}u+I).$ The analog of Theorem \ref{thm:(General-case) intro}
also holds with a universal constant constant $C$ only depending
on the dimension $n$ and $\delta^{-1}.$ 
\end{thm}

\begin{proof}
As before it is enough to consider the case when $g=1.$ Setting $M_{\C}:=(\C/\Z+i\Z)^{n}$
(which defines an Abelian variety) and identifying the convex function
$4\pi\phi(x)$ on $\R^{n}$ with a metric on the theta line bundle
$L\rightarrow M_{\C}$ \cite{g-h} this is proved as in the proof
of Theorem \ref{thm:thm H1 intro} (and is, in fact, considerably
simpler as no extension and compactification argument is needed).
In the language of $\omega_{0}-$psh functions this equivalently means
that the quasi-convex function $u(x)$ on $M$ is identified with
the $\omega_{0}-$psh function $4\pi u(z)$ on $M_{\C},$ where $\omega_{0}$
is the standard invariant Kähler form on $M_{\C},$ defined by formula
\ref{eq:def of omega noll}. Then formula \ref{eq:I on toric var}
holds on $M_{\C}$ with $\phi_{i}$ replaced by $u_{i},$ as follows
directly from Stokes theorem on $M_{\C}.$ This implies the general
case by a regularization argument, which is particularly simple in
this setting as one can use ordinary convolutions. Finally, the estimate
of $W_{1}(\mu_{0},\mu_{h})$ proceeds precisely as before, using that
any quasi-convex function on $M$ has Lipschitz constant bounded from
above by $\sqrt{n}$ (see \cite[Lemma 9]{hu}). 
\end{proof}
An analog of Theorem \ref{thm:computional geometr} also holds in
the periodic setting. In fact, the situation is facilitated by the
fact that the solution $\phi_{h}$ is piecewise affine on all of $\R^{n}$
(since the convex hull of the corresponding periodic point-cloud covers
all of $\R^{n}).$ As a consequence, the discrete Monge-Ampère equation
\ref{eq:f i eqial integral over F i} may, in the periodic setting,
be reformulated as follows. Identify a vector $\boldsymbol{\phi}\in\R^{N}$
with a quasi-periodic discrete function on the periodic discrete subset
$\Lambda_{N}$ of $\R^{n}$ determined by the point-cloud $\{x_{1},...,x_{N}\}.$
We will say that the discrete function $\boldsymbol{\phi}$ is \emph{convex}
if its graph over $\Lambda_{N}$ coincides with the boundary vertices
of its convex hull in $\R^{n+1}.$ Then the Monge-Ampère equation
for the quasi-periodic function $\phi_{h}$ on $\R^{n}$ is (almost
tautologically) equivalent to the following discrete Monge-Ampère
equation for a discrete quasi-periodic convex function $\boldsymbol{\phi}_{h}:$
\[
MA_{g}(\boldsymbol{\phi}_{h})=\boldsymbol{f},
\]
 where the discrete Monge-Ampère operator $MA_{g}(\boldsymbol{\phi})(x_{i})$
is defined as $MA_{g}(P\boldsymbol{\phi})\{x_{i}\}$ where $P\boldsymbol{\phi}$
is the function on $\R^{n}$ defined by the convex envelope of $\boldsymbol{\phi.}$
Since $P\boldsymbol{\phi}$ is piecewise affine on $\R^{n}$ this
means that $MA_{g}(\boldsymbol{\phi})(x_{i})$ is the volume (with
respect to to $gdy)$ of the convex hull of the gradients of the affine
functions representing $P\boldsymbol{\phi}$ close to $x_{i}$ (as
in point 6-8 in Section \ref{subsec:Recap-of-the}).
\begin{rem}
The discrete Monge-Ampère equation above can be seen as a quasi-periodic
variant of the Oliker-Prussner discretization of the Dirichlet problem
for the Monge-Ampère operator (where $g=1$ is assumed) \cite{o-p,n-z1,n-z2}.
\end{rem}

\end{document}